\documentclass[11pt]{amsart}
\usepackage{amsmath, amsfonts, amsthm}
\usepackage{amssymb}
\usepackage{mathrsfs}
\usepackage[a4paper, margin=1in]{geometry}
\usepackage{tikz-cd}
\usepackage{quiver}
\usepackage{enumitem}
\usepackage{tikz}
\usepackage{circuitikz}
\usepackage{subfigure}

\usepackage{hyperref}
\hypersetup{
    colorlinks=true,
    linkcolor=blue,
    filecolor=magenta,      
    urlcolor=cyan,
    citecolor=red
}
\newcommand{\Mod}[1]{\ (\mathrm{mod}\ #1)}

\newtheorem{theorem}{Theorem}[section]
\newtheorem{lemma}[theorem]{Lemma}
\newtheorem{proposition}[theorem]{Proposition}
\newtheorem{corollary}[theorem]{Corollary}

\theoremstyle{remark}
\newtheorem{remark}[theorem]{Remark}

\theoremstyle{definition}
\newtheorem{definition}[theorem]{Definition}
\newtheorem{example}[theorem]{Example}
\newtheorem{notation}[theorem]{Notation}
\newtheorem*{claim}{Claim}

\newcommand{\Acal}{\mathcal{A}}
\newcommand{\Bcal}{\mathcal{B}}
\newcommand{\Mcal}{\mathcal{M}}
\newcommand{\Ecal}{\mathcal{E}}
\newcommand{\Ocal}{\mathcal{O}}
\newcommand{\Pcal}{\mathcal{P}}

\newcommand{\Ccal}{\mathcal{C}}
\newcommand{\Scal}{\mathcal{S}}

\newcommand{\Rcal}{\mathcal{R}}

\newcommand{\PP}{\mathbb{P}}

\newcommand{\CC}{\mathbb{C}}

\newcommand{\ZZ}{\mathbb{Z}}
\newcommand{\QQ}{\mathbb{Q}}
\newcommand{\EE}{\mathbb{E}}

\newcommand{\VV}{\mathbb{V}}
\newcommand{\HH}{\mathbb{H}}

\newcommand{\Sfrak}{\mathfrak{S}}

\newcommand{\xbf}{\mathbf{x}}
\newcommand{\pbf}{\mathbf{p}}

\newcommand{\rarr}{\rightarrow}

\newcommand{\mg}{\Mcal_g}

\newcommand{\rg}{\Rcal_g}
\newcommand{\rgbar}{\overline{\Rcal}_g}
\newcommand{\rgm}{\Rcal_{g;m}}
\newcommand{\rgmbar}{\overline{\Rcal}_{g;m}}
\newcommand{\rgr}{{\Rcal}_{g, 2r}}
\newcommand{\rgrbar}{\overline{\Rcal}_{g, 2r}}
\newcommand{\rgrm}{\Rcal_{g, 2r;m}}
\newcommand{\rgrmbar}{\overline{\Rcal}_{g, 2r;m}}

\newcommand{\mgn}{\Mcal_{g, n}}

\newcommand{\ram}{\mathrm{ram}}
\newcommand{\tr}{\mathrm{tr}}

\newcommand{\ntr}{\mathrm{ntr}}
\newcommand{\exc}{\mathrm{exc}}

\newcommand{\st}{\mathrm{st}}
\newcommand{\sing}{\mathrm{sing}}
\newcommand{\CH}{\mathrm{CH}}
\renewcommand{\H}{\mathrm{H}}

\renewcommand{\Im}{\mathrm{Im}}

\newcommand{\Pic}{\mathrm{Pic}}

\newcommand{\Aut}{\mathrm{Aut}}

\newcommand{\Id}{\mathrm{Id}}

\newcommand{\GL}{\mathrm{GL}}
\newcommand{\SL}{\mathrm{SL}}
\newcommand{\Sym}{\mathrm{Sym}}

\newcommand{\R}{\mathrm{R}}
\newcommand{\Gal}{\mathrm{Gal}}
\newcommand{\Adm}{\mathrm{Adm}}
\newcommand{\rt}{\mathrm{rt}}
\newcommand{\mc}{\mathcal}
\newcommand{\ra}{\rightarrow}

\newcommand{\ol}{\overline}
\newcommand{\smallcoprod}{\mathbin{\vcenter{\hbox{$\scriptstyle\coprod$}}}}

\begin{document}

\author{Bogdan Carasca \& Riccardo Redigolo}
\address{\vskip 1.5em Humboldt Universität zu Berlin, Unter den Linden 6, 10117 Berlin}
\email{bogdanpetru.carasca@gmail.com}

\address{Humboldt Universität zu Berlin, Unter den Linden 6, 10117 Berlin}
\email{riccardo.redigolo@hu-berlin.de}

\date{}

\begin{abstract}
    We denote by $\Rcal_{g;m}$ the moduli space of $m$--pointed Prym curves of genus $g$, that is, tuples $[\widetilde C / C; x_1, \dots, x_m]$ where $[C, x_1, \dots, x_m]$ is an $m$--pointed genus $g$ curve and $\widetilde C/ C$ is an étale double cover of $C$. In this paper, we address the problem of the non--tautology of the Chow ring of $\Rcal_{g;m}$. The locus which allows us to achieve earlier bounds for this when compared to $\mg$ is the component $\Rcal\Bcal_g^0$ of the locus of bi--elliptic Prym curves. This component parametrises covers $[\widetilde C/ C]$ such that, if $C \rarr E$ is the bi--elliptic structure, the composition $\widetilde C \rarr E$ factors through an elliptic cover of $E$. Our main contribution is thus the non--tautology of the class $[\Rcal\Bcal_g^0] \in \CH^\bullet(\Rcal_g)$ when $g \ge 8$ and $g$ is even. In the course of establishing this theorem, a similar result for the compact moduli spaces $\overline{\Rcal}_{g; 2m}$ for $g + m \ge 8$ is proven.
\end{abstract}

\title[Non--tautological cycles on Prym moduli spaces] {Non--tautological cycles on Prym moduli spaces}
\maketitle

\let\thefootnote\relax

\section{Introduction}

In the landmark paper \cite{mumford1983towardEnumerativeGeometryMg}, Mumford raised for the first time the question whether the Chow ring of $\mg$ could be generated by tautological classes. Recall that the tautological ring $\R^\bullet(\mg)$ is the $\QQ$--subalgebra of $\CH^\bullet(\mg)$ generated by the $\kappa$--classes $\kappa_i := f_*(c_i(\omega_f)^{i + 1})$, where $f : \Ccal_g \rarr \mg$ is the universal curve over $\mg$ and $\omega_f$ is the relative dualising sheaf of $f$.

\vskip 0.5em
Since then, significant strides towards a better understanding of the Chow rings $\CH^\bullet(\mg)$, for low $g$, have been made:
\begin{itemize}[label=--]
    \item Mumford determines the Chow ring $\CH^\bullet(\overline\Mcal_2)$ in \cite{mumford1983towardEnumerativeGeometryMg}.
    \item Faber determines the Chow rings $\CH^\bullet(\overline\Mcal_3)$ and $\CH^\bullet(\Mcal_4)$ in \cite{faber199ChowRingM3} and \cite{faber1990chowM4}.
    \item Izadi determines the Chow ring $\CH^\bullet(\Mcal_5)$ in \cite{izadi1995chowM5}.
    \item Vakil and Penev determine $\CH^\bullet(\Mcal_6)$ in \cite{vakil2015chowRingM6}.
    \item Canning and H. Larson determine $\CH^\bullet(\Mcal_g)$ for $g = 7, 8, 9$ in \cite{larson2024chowMg}.
\end{itemize}
In all the cases above, the Chow ring is generated by tautological classes and, using work by Faber in \cite{faber1999conjDescTautMg}, the ideal of relations is fully determined.

\vskip 0.5em
The natural question is whether this pattern comes to an end as the genus $g$ increases. The answer, as expected, is affirmative. This expectation is confirmed for the first time in \cite{pandhairpande2003nontaut} where, for $g = 2(2h + 1)$, the Chow ring $\CH^\bullet(\mg)$ is shown to be non--tautological for $h$ large enough. Furthermore, van Zelm showed in \cite{zelm2018nontautBiell} that the class of the bi--elliptic locus in $\Mcal_{12}$ is non--tautological. By \cite{larson2024chowMg} and \cite{canning2025biellGen11}, this is the first instance of $[\Bcal_{g}]\notin \R^\bullet(\mg)$.

\vskip 0.5em
Another moduli space whose Chow ring has been heavily investigated in recent years is $\Acal_g$, the moduli space of principally polarised abelian varieties of dimension $g$. In \cite{pandharipande2025nontautA6} and building on methods from \cite{pandharipande2025tautProj}, they show that the class of the product locus $[\Acal_1 \times \Acal_5]$ does not lie in the tautological ring $\R^\bullet(\Acal_6)$. Recall from \cite{geerTautAg1999} that $\R^\bullet(\Acal_g)$ is the $\QQ$--subalgebra of $\CH^\bullet(\Acal_g)$ generated by the classes $\lambda_i = c_i(\EE_g)$ where $\EE_g \rarr \Acal_g$ is the Hodge bundle.

\vskip 0.5em
In this work, we focus on the Prym moduli spaces $\rg$ which parametrise \'etale double covers of curves of genus $g$. These moduli spaces were classically used to uniformise moduli of abelian varieties via the Prym map $\Pcal_g : \Rcal_g \rarr \Acal_{g - 1}$. The morphism $\Pcal_g$ is particularly effective in genus $g \le 6$, when it is dominant. Studying the Chow ring of $\rg$, as well as the Chow ring of moduli spaces of pointed Prym curves $\rgm:= \Rcal_g\times_{\mg}\Mcal_{g, m}$, is therefore very natural. Our results are the following.


\begin{theorem}\label{thm: classes RBbar are non-taut}
    Assume $g \ge 2$ and $g + m \ge 8$. The even cohomology of $\overline\Rcal_{g; 2m}$ contains non--tautological algebraic cycles.
\end{theorem}

Over the open locus $\Rcal_{g;2m} \subset \ol{\Rcal}_{g;2m}$, the following holds.
\begin{theorem}\label{thm: main non--tautology theorem}
    Assume $g \ge 2$ and $g + m\ge 8$ to be even. The even cohomology of $\Rcal_{g;2m}$ contains non--tautological algebraic cycles.
\end{theorem}

Theorem \ref{thm: classes RBbar are non-taut} and Theorem \ref{thm: main non--tautology theorem} showcase how differently $\rg$ behaves compared to $\mg$. Particularly worth mentioning are the cases $g = 8$ and $g = 10$ where the class of the bi--elliptic locus $\Bcal_g \subset \mg$ vanishes by \cite{larson2024chowMg}. Moreover, the new found bounds in the pointed case $\rgmbar$ are also lower compared to $\Mcal_{g, m}$. Recall that, from \cite{canning2022mgStable} and the references therein, $\overline{\Mcal}_{g, m}$ has non--tautological even cohomology as depicted in the table below:
\begin{table}[ht]
\centering 
\renewcommand{\arraystretch}{1.3}

\begin{tabular}{|c|c|c|c|c|c|c|}
\hline
The genus $g$ & 1 & 2 & 3 & 4 & 5 & 6 \\ \hline
$m \ge m_1(g)$ for $\overline\Mcal_{g,m}$ & $m \ge 11$ & $m \ge 20$ & $m \ge 18$ & $m \ge 16$ & $m \ge 14$ & $m \ge 12$\\ \hline
$m \ge m_2(g)$ for $\rgmbar$ & $m \ge 7$ & $m \ge 12$ & $m \ge 10$ & $m \ge 8$ & $m \ge 6$ & $m \ge 4$ \\ \hline
\end{tabular}

\vskip 0.5em
\caption{$\overline{\Mcal}_{g, m}$ has $\CH^\bullet \ne \R^\bullet$ for $m \ge m_1(g)$, $\rgmbar$ has $\CH^\bullet \ne \R^\bullet$ for $m \ge m_2(g)$}
\label{tab:my_table} 
\end{table}

Let us now provide some considerations which give context to our theorems. From \cite{farkasLudwig2010prym}, \cite{farkasVerraThetaNikulin}, \cite{farkasVerraR9}, and \cite{brunsR15}, we now know that $\rg$ is unirational for $g \le 7$, uniruled for $g \le 9$, and of general type for $g \ge 13$, $g \ne 16$. Seeing that the birational complexity of $\rg$ increases with the genus, we expect the Chow ring of $\rg$ to become more complicated. In a way, this paper confirms this heuristic reasoning. We also point out that, despite the fact that $\Rcal_8$ is uniruled, the Chow ring $\CH^\bullet(\Rcal_8)$ is not tautological.

\vskip 0.5em
Another aspect of the geometry of $\rg$ which motivated our work is the following observation. It often happens that the preimage $u^{-1}(Z) \subset \rg$ of a subvariety $Z\subset \mg$ under the forgetful map $u : \rg \rarr \mg$ is reducible. For example, this is the case when $Z$ is the hyperlliptic or the bi--elliptic locus in $\mg$. This last locus is instrumental in the proofs of Theorem \ref{thm: classes RBbar are non-taut} and Theorem \ref{thm: main non--tautology theorem}. For more on the geometry of bi--elliptic Prym curves, see \cite{naranjo1992biellPrym} or \cite{debarre1988duke}. Since, as we have mentioned, the bi--elliptic locus in $\Mcal_{12}$ is non--tautological, any component of $u^{-1}(\Bcal_{12})$ has non--tautological class in $\CH^\bullet(\Rcal_{12})$. The question now is whether this could happen at an earlier stage. As it is transparent from the proofs of Theorem \ref{thm: classes RBbar are non-taut} and Theorem \ref{thm: main non--tautology theorem}, this strategy is indeed successful.

\subsection{Structure of the paper}

The paper is organised as follows. Section \ref{section: preliminaries} discusses some basic notions and results foundational to the proof of Theorem \ref{thm: classes RBbar are non-taut} and Theorem \ref{thm: main non--tautology theorem}. The main focus of this section are the boundary strata of $\rgmbar$. Section \ref{section: cohomology R1m} is dedicated to the cohomology of $\overline\Rcal_{1;m}$. Finally, in Section \ref{section: proof of thm 11}, the results gathered so far are used to prove our main theorems.

\subsection{Notation}\label{subsection: notation}

Here is some notation used throughout this work: 
\begin{itemize}
    \item $\rgm$: The fibre product $\rg \times_{\mg} \Mcal_{g;m}$.
    \item $\rgrm$: The moduli space of tuples $[\widetilde C / C, (p_i)_{i = 1}^{2r}; (x_j)_{j = 1}^m]$ such that $\widetilde C / C$ is a double cover with branch divisor $p_1 + \dots + p_{2r}$.
    \item $\rgrmbar$: The Cornalba compactification of $\rgrm$.
    \item $\Bcal_{g, n, 2m}$: The locus in $\Mcal_{g, n+2m}$ of bi--elliptic curves $[C, (x_i)_{i = 1}^n, (y_i)_{i = 1}^{2m}]$ such that the bi--elliptic involution fixes $(x_i)_{i = 1}^n$ and sends $y_{2i - 1} \mapsto y_{2i}$ for all $1 \le i \le m$.
    \item $\Rcal\Bcal_g$: The bi--elliptic locus in $\rg$.
    \item $\{\Rcal\Bcal_g^t\colon t= 0, \dots, \lfloor \frac{g - 1}{2}\rfloor\} \cup\{\Rcal\Bcal_g'\}$: The irreducible components of the bi--elliptic locus in $\rg$.
    \item $\Rcal\Bcal^t_{g;m}$ and $\Rcal\Bcal_{g;m}'$: The pullbacks of $\Rcal\Bcal_g^t$ and $\Rcal\Bcal_g'$ to $\rgm$ under the forgetful map $\rgm \rarr \rg$.
    \item $\Rcal\Bcal_{g, n, 2m}^0$: The pre--image of $\Bcal_{g, n, 2m}$ in $\Rcal\Bcal^0_{g;n + 2m}$ under the map $\Rcal\Bcal_{g; n + 2m}^0 \rarr \Mcal_{g, n + 2m}$.
    \item $\overline\Adm(g, h)_{2m}$: The moduli space of admissible double covers $[f : C \rarr D, y_1, \dots, y_{2m}]$ such that $g(C) = g$ and $g(D)= h$, and the involution sends $y_{2i - 1} \mapsto y_{2i}$.
\end{itemize}

\subsection*{Acknowledgements}

B.C. was partially supported by the Deutsche Forschungsgemeinschaft (DFG, German Research
    Foundation) under Germany's Excellence Strategy – The Berlin Mathematics
    Research Center MATH+ (EXC-2046/1, EXC-2046/2, project ID: 390685689).

\vskip 0.5em
R.R. was supported by the DFG Research Training Group 2965 ``From geometry to numbers" project number 512730679 involving Humboldt-Universität zu Berlin and Leibniz Universität Hannover.

\vskip 0.5 em
We thank our advisor Gavril Farkas for his guidance and support. We also thank Marian Aprodu, Samir Canning, and Vlad Robu for helpful conversations.

\section{Preliminaries}\label{section: preliminaries}

This section gathers some generalities on the compactified Prym moduli spaces $\rgbar$ and pointed variants of these $\overline{\Rcal}_{g, 2r;m}$. The main scope is a to provide a combinatorial description of the boundary $\partial\Rcal_{g, 2r;m}$ analogous to $\partial\mgn$.

\subsection{Generalities on $\rgrmbar$}

In this subsection, we set some notations and basic definitions regarding the moduli stacks $\rgm$ and their natural Cornalba compactification $\rgmbar$. For further details, see \cite{cornalba1989theta} and \cite{farkasLudwig2010prym}.
\begin{definition}
    A \textit{quasi--stable $m$--pointed curve} is given by $[X, x_1, \dots, x_m]$ where $X$ is obtained by blowing up once some of the nodes of a stable $m$--pointed curve. The resulting rational components over the blown up nodes are called \textit{exceptional components}.
\end{definition}

\begin{definition}
    A \textit{stable $(2r, m)$--pointed Prym curve} is the data of $[X, (p_i)_{i = 1}^{2r}; (x_j)_{j = 1}^m, \eta, \beta]$ where $[X, (p_i)_{i = 1}^{2r}, (x_j)_{j=1}^m]$ is a quasi--stable $(2r+m)$--pointed curve, $\eta \in \Pic(X)$ is a line bundle of total degree $-r$ such that $\eta|_R = \Ocal_R(1)$ for every exceptional component $R$ of $X$, and ${\beta : \eta^{\otimes 2} \rarr \Ocal_X(-p_1 - \dots - p_{2r})}$ is a homomorphism which is generically non--zero on non--exceptional components.
\end{definition}
We denote by $\rgrmbar$ the moduli stack of stable $(2r, m)$--pointed Prym curves, and by $\rgrm \subset \rgrmbar$ the locus parametrising smooth curves. By \cite{budRg2}, the stacks $\rgrmbar$ are all irreducible. We use the notation $\rgrbar := \overline{\Rcal}_{g, 2r; 0}$ and $\rgmbar:= \overline{\Rcal}_{g, 0;m}$, and a similar convention is maintained for $\rgr$ and $\rgm$.

\begin{remark}
    If $[X, (p_i)_{i = 1}^{2r}; (x_j)_{j = 1}^m, \eta]$ is a stable $(2r, m)$--pointed Prym curve, we have that for any non--exceptional component $C$ of $X$, $\eta|_C^{\otimes 2} \cong \Ocal_C\left(-\sum_{p_i \in C} p_i\right)$. This square root is equivalent to a ramified double cover of $C$, branched at the points $\{p_i\}_{i = 1}^{2r} \cap C$. This observation relates Cornalba's perspective on $\rgrmbar$ to Beauville's in \cite{beavilleSchottky} which uses admissible covers. This dictionary is explained very well in \cite{farkasLudwig2010prym}. We will often switch perspectives between these two standpoints.
\end{remark}

\begin{definition}
    We let $\rgrm'$ denote the moduli stack of tuples $[\widetilde C/C, (p_i)_{i = 1}^{2r}; (x_j)_{j = 1}^m, (x_j^\pm)_{j = 1}^m]$ where $[\widetilde C/C, (p_i)_{i = 1}^{2r}; (x_j)_{j = 1}^m] \in \rgrm$ and $\{x_j^\pm\}$ is the fibre of $\widetilde C \rarr C$ over $x_j$. We denote by $\overline\Rcal_{g, 2r;m}'$ the Cornalba compactification of $\Rcal_{g, 2r;m}'$. We point out that, in $\Rcal_{g, 2r;m}'$, the ordering of the fibres of $\widetilde C \rarr C$ over the $x_j$'s is part of the data.
\end{definition}

The irreducible components of the Prym bi--elliptic locus $\Rcal\Bcal_g \subset \rg$ have been described in \cite{naranjo1992biellPrym} and \cite{debarre1988duke} as follows. Let $[\widetilde C / C] \in \Rcal\Bcal_g$ with bi--elliptic structure $f : C \rarr E$. There are two possibilities for $\Gal(\widetilde C / E)$: either $\Gal(\widetilde C / E) = \ZZ/2$ or $\Gal(\widetilde C/ E) = \ZZ/2 \times \ZZ/2$. The component $\Rcal\Bcal_g'$ is prescribed by the condition $\Gal(\widetilde C/E) = \ZZ/2$.  For $0 \le t \le \lfloor \frac{g - 1}{2}\rfloor$, the component $\Rcal\Bcal_g^t$ is given by $g(C_1) = t + 1$ and $g(C_2) = g - t$ where $C_1$ and $C_2$ are the two other intermediate subcovers. The preimage of the component $\Rcal\Bcal_g^t$ (respectively $\Rcal\Bcal_g'$) inside $\Rcal_{g;m}$ under the forgetful map $\rgm \rarr \rg$ is denoted by $\Rcal\Bcal^t_{g;m}$ (respectively $\Rcal\Bcal_{g;m}'$). 

\vskip 0.5em
We let $\Rcal\Bcal_{g, n, 2m}^t$ denote the locus in $\Rcal_{g; n + 2m}$ of $[\widetilde C/ C, (x_i)_{i = 1}^n, (y_j)_{j = 1}^{2m}]$ such that $[\widetilde C/C] \in \Rcal\Bcal_g^t$, the bi--elliptic involution fixes the $x_i$, and sends $y_{2j -1}$ to $y_{2j}$.

\subsection{Pullbacks of Boundary Strata}

\vskip 0.5 em
The goal of this subsection will be to prove that pullbacks of tautological classes of $\rgmbar$ by the Prym glueing and clutching morphisms admit a tautological K\"unneth decomposition.
As in \cite[Appendix A]{pandhairpande2003nontaut}, we begin by noticing that the boundary strata of $\ol{\mc R}_{g;m}$ can be described explicitly in terms of combinatorial data.
\begin{definition}
    A \textit{weighted graph} is a tuple $G:=(V(G),\ H(G),\ E(G),\ L(G), g, \iota,a)$ where
	\begin{enumerate}[label=$\roman*)$]
		\item $V(G)$ is a vertex set equipped with a genus function $g:V(G)\ra \mathbb Z_{\geq 0}$;
		\item $H(G)$ is a half--edge set equipped with an involution $\iota$ and a vertex assignement \[a:H(G)\ra V(G);\] 
		\item $L(G)$ is a set of legs identified with the fixed points of $\iota$;
		\item $E(G)$ is a set of edges identified with $\left(H(G)\backslash L(G)\right)/\iota$	
	\end{enumerate}
\end{definition}

Given a graph $G$ and a vertex $v$ we denote its \textit{valence} by $n(v):=\#a^{-1}v$. 
\begin{definition}
	A \textit{finite harmonic morphism} of degree $d$ between two weighted graphs $G'$ and $G$ is the data of $\phi = (\phi_V, \phi_H, \{n_e\}_{e \in E(G)}): G' \ra G$ satisfying the following properties.
	\begin{enumerate}[label=$\roman*)$]
		\item $\phi_V:V(G')\ra V(G)$ is surjective.
		\item $\phi_H:H(G')\ra H(G)$ is surjective and makes the following diagrams commute:
		\[
		\begin{tikzcd}
			H(G') \arrow[d, "\phi_H"] \arrow[r, "\iota"] & H(G') \arrow[d,"\phi_H"]\\
			H(G) \arrow[r, "\iota"] & H(G)\\
		\end{tikzcd}
		\qquad 
		\begin{tikzcd}
			H(G') \arrow[d, "\phi_H"] \arrow[r, "a"] & V(G') \arrow[d,"\phi_V"]\\
			H(G) \arrow[r, "a"] & V(G).\\
		\end{tikzcd}\]
		\item The following quantity 
		\[\deg(\phi):=\sum_{e'\in \phi_E^{-1}(e)} n_{e'}\]
        is equal to $d$ independently of $e\in E(G)$.
        \item For each $l\in L(G)$, $\deg(\phi) = \#\phi_H^{-1}(l)$.
		\item For each $v\in V(G)$, we have
		\[\sum_{v'\in \phi_V^{-1}(v)} \left(2g(v')-2\right)=\deg(\phi)\left(2g(v)-2\right)+\sum_{v'\in \phi_V^{-1}(v)}\sum_{h'\in a^{-1}(v')} \left(n_{[h']}-1\right).\]
	\end{enumerate}
	The morphism $\phi$ is said to be \textit{\'etale} if for each $v'\in V(G'),\ h'\in a^{-1}(v')$ we have $n_{[h']}=1$.
\end{definition}

\begin{remark}
	More geometrically, we may look at finite harmonic morphisms as the combinatorial counterparts of morphisms of degree $d$ between quasi-stable curves that, on each component, are ramified only at the nodes.
\end{remark}

Recall also that given a graph $G$ we have the following notion of stability.
\begin{definition}
	A graph $G$ is said to be \textit{semi--stable} if and only if it is connected and for each $v\in V(G)$, we have
	\[2g(v)-2+n(v)\geq 0\]
	If the inequality above is strict for each $v \in V(G)$, we say that $G$ is \textit{stable}.
\end{definition}
\begin{definition}\label{def: prym structure}
	Let $G'$ and $G$ be semi--stable graphs and let $\phi:G' \ra G$ be a finite harmonic morphism of degree $2$. We say that $\phi$ is a \textit{Prym structure on $G$} if and only if it satisfies the following properties.
	\begin{enumerate}[label=$\roman*)$]
        \item For each $v\in V(G)$ such that $2g(v)-2+n(v)=0$, $h\in a^{-1}(v)$, we have that \[2g(a(\iota(h)))-2+n(a(\iota(h)))>0.\] 
		\item For each $v\in  V(G)$ such that $2g(v)-2+n(v)>0$, \[\#\{h\in a^{-1}(v)\ |\ 2g(a(h))-2+n(a(h))=0\}\equiv 0 \Mod{2}.\]
		\item For each $h'\in H(G')$, $n_{[h']}\leq 2$ with equality holding if and only if $h'\notin L(G')$ and either $a(h')$ or $a(\iota(h'))$ satisfies $2g(\cdot)-2+n(\cdot)=0$.
	\end{enumerate} 
\end{definition}

\begin{remark}
    Notice that conditions $i),ii),iii)$ are the combinatorial counterparts of the following properties of a degree $2$ morphism between semistable curves.
    \begin{enumerate}[label=$\roman*)$]
        \item No two strictly semistable components of the base meet.
		\item Every stable component has an even number of exceptional nodes.
		\item The morphism is ramified only at the nodes meeting the exceptional components.
	\end{enumerate} 
\end{remark}
\begin{notation}
    If $\phi:G'\ra G$ is a Prym structure on $G$, we will denote by $V^{\exc}(G)$ the set of strictly semistable vertices, by $V^{\ram}(G)$ the set of stable vertices with an edge connected to an unstable vertex, by $V^{\tr}(G)$ the set of vertices with $2$ preimages and by $V^{\ntr}(G)$ the set of non--exceptional vertices with one preimage. Given $v\in V^{\ram}(G)$, we denote by $E_v^{\ram}(G)$ the set of edges connecting $v$ to unstable vertices. We will also denote $2r(v):=\#E_v^{\ram}(G)$ and $m(v):=n(v)-2r(v)$.
\end{notation}

\begin{definition}\label{def: dual harmonic map prym curve}
    We \hspace{0.5 pt} define \hspace{0.5 pt} the \hspace{0.5 pt} \textit{dual \hspace{0.5 pt} harmonic \hspace{0.5 pt} map} \hspace{0.5 pt} of \hspace{0.5 pt} a \hspace{0.5 pt} stable \hspace{0.5 pt} $m$--pointed \hspace{0.5 pt} Prym \hspace{0.5 pt} curve ${[f: \widetilde C \rarr C; (x_j)_{j = 1}^m]}$ as the map between the dual graphs of $\widetilde C$ and $C$, according to how $f$ maps irreducible components to irreducible components and special points to special points. For each $e \in E(\Gamma(\widetilde C))$, $n_e$ is given by the ramification index of $f$ at the corresponding special point.
\end{definition}

Notice that a boundary stratum of $\ol{\mc R}'_{g;m}$ naturally determines a Prym structure on a semistable graph by considering the dual harmonic map at its generic point.

\vskip 0.5em
Conversely, given a Prym structure $\phi:G'\ra G$, there exists a boundary stratum $\Delta_\phi\subset \ol{\mc R}'_{g;m}$ whose generic point has $\phi$ as the dual harmonic map. Indeed, defining for each Prym structure $\phi$, the stack
\[\ol{\mc R}'_{\phi}:=\prod_{v\in V^{ntr}(G')} \ol{\mc R}'_{g(v);n(v)}\times \prod_{v\in V^{tr}(G')} \ol{\mc M}_{g(v);n(v)}\times \prod_{v\in V^{exc}(G')} \ol{\mc R}'_{g(v),2r(v);m(v)},\]
we take $\Delta_\phi$ to be the image of the gluing map $\chi_\phi:\ol{\mc R}'_{\phi}\ra \ol{\mc R}'_{g;m}$.

\vskip 0.5 em
Slightly modifying Definition \ref{def: prym structure} to account for ramification at the $2r$ marked points, one obtains an analogous description of the boundary stratification of $\rgrmbar'$. For future use, we fix the notation for the gluing morphisms that give the boundary components of $\ol{\mc R}'_{g;m}$.

\begin{example}\label{ex: gluing and clutching maps}
    The gluing morphisms giving the boundary components of $\rgmbar'$ correspond to the following Prym structures.
    Let $0 \le i \le \lfloor\frac{g - 1}{2}\rfloor$ and $S \subset \{p_k\}_{k = 1}^{2r} \cup \{x_j\}_{j = 1}^m$ be such that $2i - 2 + \# S > 0$ and $2(g - i) - 2 + \# S^c > 0$. We let $S_\pbf := S \cap \{p_i\}_{i = 1}^{2r}$ and $S_{\xbf} := S \cap \{x_j\}_{j = 1}^m$.
    \begin{enumerate}
    \item
    Assume first $\#S_\pbf$ is even. Define the gluing map
    \[
    \overline\Rcal'_{i, S_\pbf;S_\xbf \cup \{*\}} \times \overline{\Rcal}'_{g - i, S_\pbf^c; S_\xbf^c \cup \{*\}} \rarr \rgrmbar'
    \]
    which glues the two curves at the special marked point $*$, and the two double covers along the preimages of the marked point $*$. See Figure \ref{fig: gluing 1 2}.1.
    \item Assume next that $\#S_\pbf$ is odd. Define the gluing map
    \[
    \overline\Rcal'_{i, S_{\pbf} \cup \{*\}; S_\xbf} \times \overline\Rcal'_{g - i, S_\pbf^c\cup \{\star\}; S_\xbf} \rarr \overline\Rcal_{g, 2r;m}'
    \]
    which joins the given two curves by an exceptional component at the markings $*$ and $\star$. 
    \begin{figure}[!ht]
    \hskip -2 em
		\begin{minipage}{0.225\textwidth}
			\begin{center}
				\begin{circuitikz}
        \tikzstyle{every node}=[font=\fontsize{9.1pt}{11.8pt}\selectfont]
        \draw (9.375,12.75) to[short] (10.875,12.75);
        \node [font=\fontsize{9.1pt}{11.8pt}\selectfont, inner xsep=0.080cm, inner ysep=0.085cm, rounded corners=0.020cm] at (9.125,12.5) {$g-i$};
        \node [font=\fontsize{9.1pt}{11.8pt}\selectfont, inner xsep=0.080cm, inner ysep=0.085cm, rounded corners=0.000cm] at (10.75,12.375) {$i$};
        \begin{scope}[transparency group, opacity=1]
        \node at (9.375,12.75) [circ, color={rgb,255:red,0; green,0; blue,0}] {};
        \end{scope}
        \begin{scope}[transparency group, opacity=1]
        \node at (10.875,12.75) [circ, color={rgb,255:red,0; green,0; blue,0}] {};
        \end{scope}
        \draw (9.375,12.75) to[short] (8.5,13.625);
        \draw (9.375,12.75) to[short] (8.375,12.75);
        \draw [short] (10.875,12.75) -- (11.375,13.625);
        \draw [short] (10.875,12.75) -- (11.625,12);
        \draw [-{Stealth[scale=1.5]}, ] (10,14.125) -- (10,13.25);
        \begin{scope}[transparency group, opacity=1]
        \node at (9,15.25) [circ, color={rgb,255:red,0; green,0; blue,0}] {};
        \end{scope}
        \begin{scope}[transparency group, opacity=1]
        \node at (10.875,15.25) [circ, color={rgb,255:red,0; green,0; blue,0}] {};
        \end{scope}
        \node [font=\fontsize{8.0pt}{10.4pt}\selectfont, inner xsep=0.080cm, inner ysep=0.085cm, rounded corners=0.000cm] at (12.25,15.625) {$x^-$};
        \node [font=\fontsize{6.8pt}{8.9pt}\selectfont, inner xsep=0.080cm, inner ysep=0.085cm, rounded corners=0.020cm] at (10.875,14.75) {$2i\hspace{-2pt}-\hspace{-2pt}1\hspace{-2pt}+\hspace{-2pt}r_2$};
        \draw [short] (9,15.25) .. controls (9.75,15.75) and (10.25,15.75) .. (10.875,15.25);
        \draw [short] (9,15.25) .. controls (9.75,14.75) and (10.25,14.75) .. (10.875,15.25);
        \draw [short] (9,15.25) -- (8.625,15.75);
        \draw [short] (9,15.25) .. controls (8.5,15.5) and (8.75,15.375) .. (8,15.5);
        \draw [short] (9,15.25) .. controls (8.5,15) and (8.875,15.125) .. (8,15.125);
        \draw [short] (10.875,15.25) -- (11.125,15.75);
        \draw [short] (10.875,15.25) .. controls (11.5,15.5) and (11.125,15.375) .. (12,15.5);
        \draw [short] (10.875,15.25) .. controls (11.5,15) and (11.125,15.125) .. (12,15.125);
        \node [font=\fontsize{8.0pt}{10.4pt}\selectfont, inner xsep=0.080cm, inner ysep=0.085cm, rounded corners=0.000cm] at (11.75,11.875) {$x$};
        \node [font=\fontsize{8.0pt}{10.4pt}\selectfont, inner xsep=0.080cm, inner ysep=0.085cm, rounded corners=0.016cm] at (12.25,15) {$x^+$};
        \node [font=\fontsize{8.0pt}{10.4pt}\selectfont, inner xsep=0.080cm, inner ysep=0.085cm, rounded corners=0.000cm] at (7.625,15) {$x^-$};
        \node [font=\fontsize{8.0pt}{10.4pt}\selectfont, inner xsep=0.080cm, inner ysep=0.085cm, rounded corners=0.000cm] at (11.5,13.75) {$2r_2$};
        \node [font=\fontsize{8.0pt}{10.4pt}\selectfont, inner xsep=0.080cm, inner ysep=0.085cm, rounded corners=0.000cm] at (8.125,12.75) {$x$};
        \node [font=\fontsize{6.4pt}{8.5pt}\selectfont, inner xsep=0.080cm, inner ysep=0.085cm, rounded corners=0.020cm] at (8.875,14.75) {$2g\hspace{-2pt}-\hspace{-2pt}2i\hspace{-2pt}-\hspace{-2pt}1\hspace{-2pt}+\hspace{-2pt}r_1$};
        \node [font=\fontsize{8.0pt}{10.4pt}\selectfont, inner xsep=0.080cm, inner ysep=0.085cm, rounded corners=0.016cm] at (7.625,15.625) {$x^+$};
        \node [font=\fontsize{8.0pt}{10.4pt}\selectfont, inner xsep=0.080cm, inner ysep=0.085cm, rounded corners=0.000cm] at (8.375,13.75) {$2r_1$};
        \node [font=\fontsize{8.0pt}{10.4pt}\selectfont, inner xsep=0.080cm, inner ysep=0.085cm, rounded corners=0.000cm] at (8.5,16) {$2r_1$};
        \node [font=\fontsize{8.0pt}{10.4pt}\selectfont, inner xsep=0.080cm, inner ysep=0.085cm, rounded corners=0.000cm] at (11.25,16) {$2r_2$};
        \node [font=\fontsize{8.0pt}{10.4pt}\selectfont, inner xsep=0.080cm, inner ysep=0.085cm, rounded corners=0.000cm] at (10.025,12.475) {$*$};
        \node [font=\fontsize{8.0pt}{10.4pt}\selectfont, inner xsep=0.080cm, inner ysep=0.085cm, rounded corners=0.016cm] at (10.025,15.85) {$*^+$};
        \node [font=\fontsize{8.0pt}{10.4pt}\selectfont, inner xsep=0.080cm, inner ysep=0.085cm, rounded corners=0.000cm] at (10.025,15.1) {$*^-$};
        \node [font=\fontsize{10.2pt}{13.3pt}\selectfont, inner xsep=0.080cm, inner ysep=0.085cm, rounded corners=0.010cm] at (6.75,14.25) {$\phi_1$:};
        \end{circuitikz}
			\end{center}
		\end{minipage}
        \qquad \qquad \qquad \qquad 
		\begin{minipage}{0.225\textwidth}
			\begin{center}
				\begin{circuitikz}
                \tikzstyle{every node}=[font=\fontsize{10.2pt}{13.3pt}\selectfont]
                \draw (8.625,12.75) to[short] (11.625,12.75);
                \node [font=\fontsize{9.1pt}{11.8pt}\selectfont, inner xsep=0.080cm, inner ysep=0.085cm, rounded corners=0.020cm] at (8.375,12.5) {$g-i$};
                \node [font=\fontsize{9.1pt}{11.8pt}\selectfont, inner xsep=0.080cm, inner ysep=0.085cm, rounded corners=0.000cm] at (11.5,12.375) {$i$};
                \begin{scope}[transparency group, opacity=1]
                \node at (8.625,12.75) [circ, color={rgb,255:red,0; green,0; blue,0}] {};
                \end{scope}
                \begin{scope}[transparency group, opacity=1]
                \node at (11.625,12.75) [circ, color={rgb,255:red,0; green,0; blue,0}] {};
                \end{scope}
                \draw (8.625,12.75) to[short] (7.75,13.625);
                \draw (8.625,12.75) to[short] (7.625,12.75);
                \draw [short] (11.625,12.75) -- (12.125,13.625);
                \draw [short] (11.625,12.75) -- (12.375,12);
                \draw [-{Stealth[scale=1.5]}, ] (10,14.125) -- (10,13.25);
                \begin{scope}[transparency group, opacity=1]
                \node at (9,15.25) [circ, color={rgb,255:red,0; green,0; blue,0}] {};
                \end{scope}
                \begin{scope}[transparency group, opacity=1]
                \node at (10.875,15.25) [circ, color={rgb,255:red,0; green,0; blue,0}] {};
                \end{scope}
                \node [font=\fontsize{8.0pt}{10.4pt}\selectfont, inner xsep=0.080cm, inner ysep=0.085cm, rounded corners=0.000cm] at (12.25,15.625) {$x^-$};
                \node [font=\fontsize{6.8pt}{8.9pt}\selectfont, inner xsep=0.080cm, inner ysep=0.085cm, rounded corners=0.020cm] at (10.875,14.75) {$2i\hspace{-2pt}-\hspace{-2pt}1\hspace{-2pt}+\hspace{-2pt}r_2$};
                \draw [short] (9,15.25) .. controls (9.75,15.25) and (10.125,15.25) .. (10.875,15.25);
                \draw [short] (9,15.25) -- (8.625,15.75);
                \draw [short] (9,15.25) .. controls (8.5,15.5) and (8.75,15.375) .. (8,15.5);
                \draw [short] (9,15.25) .. controls (8.5,15) and (8.875,15.125) .. (8,15.125);
                \draw [short] (10.875,15.25) -- (11.125,15.75);
                \draw [short] (10.875,15.25) .. controls (11.5,15.5) and (11.125,15.375) .. (12,15.5);
                \draw [short] (10.875,15.25) .. controls (11.5,15) and (11.125,15.125) .. (12,15.125);
                \node [font=\fontsize{8.0pt}{10.4pt}\selectfont, inner xsep=0.080cm, inner ysep=0.085cm, rounded corners=0.000cm] at (12.5,11.875) {$x$};
                \node [font=\fontsize{8.0pt}{10.4pt}\selectfont, inner xsep=0.080cm, inner ysep=0.085cm, rounded corners=0.016cm] at (12.25,15) {$x^+$};
                \node [font=\fontsize{8.0pt}{10.4pt}\selectfont, inner xsep=0.080cm, inner ysep=0.085cm, rounded corners=0.000cm] at (7.625,15) {$x^-$};
                \node [font=\fontsize{8.0pt}{10.4pt}\selectfont, inner xsep=0.080cm, inner ysep=0.085cm, rounded corners=0.000cm] at (12.25,13.75) {$2r_2\hspace{-2pt}-\hspace{-2pt}1$};
                \node [font=\fontsize{8.0pt}{10.4pt}\selectfont, inner xsep=0.080cm, inner ysep=0.085cm, rounded corners=0.000cm] at (7.375,12.75) {$x$};
                \node [font=\fontsize{6.8pt}{8.9pt}\selectfont, inner xsep=0.080cm, inner ysep=0.085cm, rounded corners=0.020cm] at (8.875,14.75) {$2g\hspace{-2pt}-\hspace{-2pt}2i\hspace{-2pt}-\hspace{-2pt}1\hspace{-2pt}+\hspace{-2pt}r_1$};
                \node [font=\fontsize{8.0pt}{10.4pt}\selectfont, inner xsep=0.080cm, inner ysep=0.085cm, rounded corners=0.016cm] at (7.625,15.625) {$x^+$};
                \node [font=\fontsize{8.0pt}{10.4pt}\selectfont, inner xsep=0.080cm, inner ysep=0.085cm, rounded corners=0.000cm] at (7.625,13.75) {$2r_1\hspace{-2pt}-\hspace{-2pt}1$};
                \node [font=\fontsize{8.0pt}{10.4pt}\selectfont, inner xsep=0.080cm, inner ysep=0.085cm, rounded corners=0.000cm] at (8.5,16) {$2r_1\hspace{-2pt}-\hspace{-2pt}1$};
                \node [font=\fontsize{8.0pt}{10.4pt}\selectfont, inner xsep=0.080cm, inner ysep=0.085cm, rounded corners=0.000cm] at (11.25,16) {$2r_2\hspace{-2pt}-\hspace{-2pt}1$};
                \node [font=\fontsize{8.0pt}{10.4pt}\selectfont, inner xsep=0.080cm, inner ysep=0.085cm, rounded corners=0.000cm] at (9.25,12.45) {$*$};
                \node [font=\fontsize{8.0pt}{10.4pt}\selectfont, inner xsep=0.080cm, inner ysep=0.085cm, rounded corners=0.000cm] at (9.375,15.375) {$*$};
                \node [font=\fontsize{10.2pt}{13.3pt}\selectfont, inner xsep=0.080cm, inner ysep=0.085cm, rounded corners=0.020cm] at (6.75,14.25) {$\phi_2$:};
                \begin{scope}[transparency group, opacity=1]
                \node at (9.875,15.25) [circ, color={rgb,255:red,0; green,0; blue,0}] {};
                \end{scope}
                \begin{scope}[transparency group, opacity=1]
                \node at (10,12.75) [circ, color={rgb,255:red,0; green,0; blue,0}] {};
                \end{scope}
                \node [font=\fontsize{8.0pt}{10.4pt}\selectfont, inner xsep=0.080cm, inner ysep=0.085cm, rounded corners=0.000cm] at (10.75,12.45) {$\star$};
                \node [font=\fontsize{8.0pt}{10.4pt}\selectfont, inner xsep=0.080cm, inner ysep=0.085cm, rounded corners=0.000cm] at (10.375,15.375) {$\star$};
                \end{circuitikz}
			\end{center}
		\end{minipage}
        \qquad
        \qquad
        \qquad 
		\caption{The Harmonic Morphisms giving rise to the gluing maps in (1) and (2)}
		\label{fig: gluing 1 2}
	\end{figure}
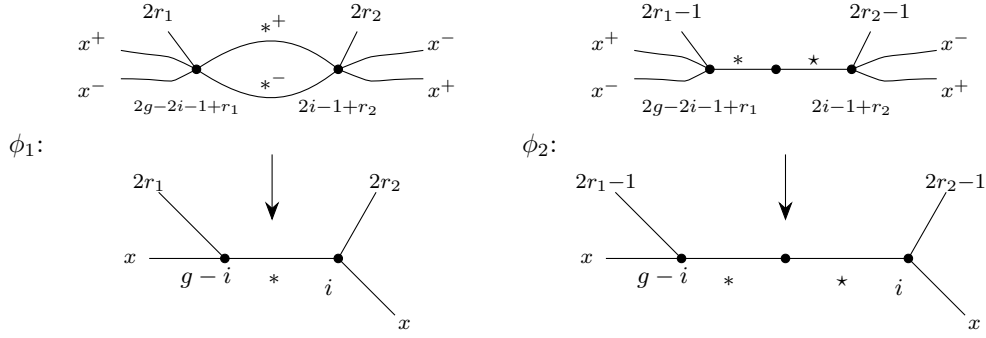
    \item If $S_\pbf = \varnothing$, to the previous list, we add the gluing map
    \[
    \overline{\Mcal}_{i, x+1} \times \overline\Rcal'_{g - i, 2r; m-x+1} \rarr \rgrmbar'
    \]
    which glues a cover in $\overline\Rcal_{g - i, 2r; m-x+1}$ to the trivial étale cover $C \amalg C  \rarr C$ of at the point $*$ for $[C, p_1,\dots, p_x,*] \in \Mcal_{i, x+1}$. See Figure \ref{fig: gluing 3}.1.           
    \item Define the clutching maps
    \begin{align*}
    &\overline\Rcal_{g-1, 2r;m + 2}' \rarr \rgrmbar', \quad [\widetilde C/ C, (p_i)_{i = 1}^{2r}; (x_j)_{j = 1}^m, (x_j^\pm)_{j = 1}^m]\mapsto\\
    &[\widetilde C/\{x_{m + 1}^\pm \sim x_{m + 2}^\pm\} \rarr C/\{x_{m + 1} \sim x_{m + 2}\}, (p_i)_{i =1}^{2r}; (x_j)_{j = 1}^m, (x^\pm_j)_{j = 1}^m].
    \end{align*}
    \begin{figure}[!ht]
    \hskip -2 em
		\begin{minipage}{0.225\textwidth}
			\begin{center}
				\begin{circuitikz}
\tikzstyle{every node}=[font=\fontsize{6.8pt}{8.9pt}\selectfont]
\draw (9.375,12.75) to[short] (10.875,12.75);
\node [font=\fontsize{9.1pt}{11.8pt}\selectfont, inner xsep=0.080cm, inner ysep=0.085cm, rounded corners=0.000cm] at (9.125,12.5) {$i$};
\node [font=\fontsize{9.1pt}{11.8pt}\selectfont, inner xsep=0.080cm, inner ysep=0.085cm, rounded corners=0.020cm] at (10.75,12.375) {$g\hspace{-2pt}-\hspace{-2pt}i$};
\begin{scope}[transparency group, opacity=1]
\node at (9.375,12.75) [circ, color={rgb,255:red,0; green,0; blue,0}] {};
\end{scope}
\begin{scope}[transparency group, opacity=1]
\node at (10.875,12.75) [circ, color={rgb,255:red,0; green,0; blue,0}] {};
\end{scope}
\draw (9.375,12.75) to[short] (8.375,12.75);
\draw [short] (10.875,12.75) -- (11.375,13.625);
\draw [short] (10.875,12.75) -- (11.625,12);
\draw [-{Stealth[scale=1.5]}, ] (10,14.125) -- (10,13.25);
\begin{scope}[transparency group, opacity=1]
\node at (9.25,14.875) [circ, color={rgb,255:red,0; green,0; blue,0}] {};
\end{scope}
\begin{scope}[transparency group, opacity=1]
\node at (11.25,15.25) [circ, color={rgb,255:red,0; green,0; blue,0}] {};
\end{scope}
\node [font=\fontsize{8.0pt}{10.4pt}\selectfont, inner xsep=0.080cm, inner ysep=0.085cm, rounded corners=0.000cm] at (12.75,15.625) {$x^-$};
\node [font=\fontsize{6.8pt}{8.9pt}\selectfont, inner xsep=0.080cm, inner ysep=0.085cm, rounded corners=0.020cm] at (11.5,14.75) {$2g\hspace{-2pt}-\hspace{-2pt}2i\hspace{-2pt}-\hspace{-2pt}1\hspace{-2pt}+\hspace{-2pt}r$};
\draw [short] (9.25,15.75) -- (11.25,15.25);
\draw [short] (9.25,14.875) -- (11.25,15.25);
\draw [short] (9.25,15.75) -- (8.25,15.75);
\draw [short] (9.25,14.875) -- (8.25,14.875);
\draw [short] (11.25,15.25) -- (11.5,15.75);
\draw [short] (11.25,15.25) .. controls (11.875,15.5) and (11.5,15.375) .. (12.5,15.5);
\draw [short] (11.25,15.25) .. controls (11.875,15) and (11.5,15.125) .. (12.5,15.125);
\node [font=\fontsize{8.0pt}{10.4pt}\selectfont, inner xsep=0.080cm, inner ysep=0.085cm, rounded corners=0.000cm] at (11.75,11.875) {$x$};
\node [font=\fontsize{8.0pt}{10.4pt}\selectfont, inner xsep=0.080cm, inner ysep=0.085cm, rounded corners=0.016cm] at (12.75,15) {$x^+$};
\node [font=\fontsize{8.0pt}{10.4pt}\selectfont, inner xsep=0.080cm, inner ysep=0.085cm, rounded corners=0.000cm] at (8,14.875) {$x^-$};
\node [font=\fontsize{8.0pt}{10.4pt}\selectfont, inner xsep=0.080cm, inner ysep=0.085cm, rounded corners=0.000cm] at (11.5,13.75) {$2r$};
\node [font=\fontsize{8.0pt}{10.4pt}\selectfont, inner xsep=0.080cm, inner ysep=0.085cm, rounded corners=0.000cm] at (8.125,12.75) {$x$};
\node [font=\fontsize{6.8pt}{8.9pt}\selectfont, inner xsep=0.080cm, inner ysep=0.085cm, rounded corners=0.000cm] at (9.25,14.5) {$i$};
\node [font=\fontsize{8.0pt}{10.4pt}\selectfont, inner xsep=0.080cm, inner ysep=0.085cm, rounded corners=0.016cm] at (8,16) {$x^+$};
\node [font=\fontsize{8.0pt}{10.4pt}\selectfont, inner xsep=0.080cm, inner ysep=0.085cm, rounded corners=0.000cm] at (11.625,16) {$2r$};
\node [font=\fontsize{8.0pt}{10.4pt}\selectfont, inner xsep=0.080cm, inner ysep=0.085cm, rounded corners=0.000cm] at (10.125,12.375) {$*$};
\node [font=\fontsize{8.0pt}{10.4pt}\selectfont, inner xsep=0.080cm, inner ysep=0.085cm, rounded corners=0.016cm] at (10.25,15.725) {$*^+$};
\node [font=\fontsize{8.0pt}{10.4pt}\selectfont, inner xsep=0.080cm, inner ysep=0.085cm, rounded corners=0.000cm] at (10.25,14.825) {$*^-$};
\node [font=\fontsize{10.2pt}{13.3pt}\selectfont, inner xsep=0.080cm, inner ysep=0.085cm, rounded corners=0.010cm] at (7.125,14.25) {$\phi_3$:};
\begin{scope}[transparency group, opacity=1]
\node at (9.25,15.75) [circ, color={rgb,255:red,0; green,0; blue,0}] {};
\end{scope}
\node [font=\fontsize{6.8pt}{8.9pt}\selectfont, inner xsep=0.080cm, inner ysep=0.085cm, rounded corners=0.000cm] at (9.125,16) {$i$};
\end{circuitikz}
                \end{center}
		\end{minipage}
        \qquad
        \qquad
        \qquad 
        \qquad 
        \begin{minipage}{0.225\textwidth}
			\begin{center}
            \begin{circuitikz}
            \tikzstyle{every node}=[font=\fontsize{8.0pt}{10.4pt}\selectfont]
            \node [font=\fontsize{9.1pt}{11.8pt}\selectfont, inner xsep=0.080cm, inner ysep=0.085cm, rounded corners=0.020cm] at (10.75,12.375) {$g\hspace{-2pt}-\hspace{-2pt}1$};
            \begin{scope}[transparency group, opacity=1]
            \node at (10.875,12.75) [circ, color={rgb,255:red,0; green,0; blue,0}] {};
            \end{scope}
            \draw [short] (10.875,12.75) -- (11.375,13.625);
            \draw [short] (10.875,12.75) -- (11.625,12);
            \draw [-{Stealth[scale=1.5]}, ] (10.875,14.125) -- (10.875,13.25);
            \begin{scope}[transparency group, opacity=1]
            \node at (10.875,15.25) [circ, color={rgb,255:red,0; green,0; blue,0}] {};
            \end{scope}
            \node [font=\fontsize{8.0pt}{10.4pt}\selectfont, inner xsep=0.080cm, inner ysep=0.085cm, rounded corners=0.000cm] at (12.375,15.625) {$m^-$};
            \node [font=\fontsize{6.8pt}{8.9pt}\selectfont, inner xsep=0.080cm, inner ysep=0.085cm, rounded corners=0.020cm] at (11.325,14.85) {$2g\hspace{-2pt}-\hspace{-2pt}3\hspace{-2pt}+\hspace{-2pt}r$};
            \draw [short] (10.875,15.25) -- (11.125,15.75);
            \draw [short] (10.875,15.25) .. controls (11.5,15.5) and (11.125,15.375) .. (12.125,15.5);
            \draw [short] (10.875,15.25) .. controls (11.5,15) and (11.125,15.125) .. (12.125,15.125);
            \node [font=\fontsize{8.0pt}{10.4pt}\selectfont, inner xsep=0.080cm, inner ysep=0.085cm, rounded corners=0.000cm] at (11.75,11.875) {$m$};
            \node [font=\fontsize{8.0pt}{10.4pt}\selectfont, inner xsep=0.080cm, inner ysep=0.085cm, rounded corners=0.016cm] at (12.375,15) {$m^+$};
            \node [font=\fontsize{8.0pt}{10.4pt}\selectfont, inner xsep=0.080cm, inner ysep=0.085cm, rounded corners=0.000cm] at (11.5,13.75) {$2r$};
            \node [font=\fontsize{8.0pt}{10.4pt}\selectfont, inner xsep=0.080cm, inner ysep=0.085cm, rounded corners=0.000cm] at (11.25,16) {$2r$};
            \node [font=\fontsize{8.0pt}{10.4pt}\selectfont, inner xsep=0.080cm, inner ysep=0.085cm, rounded corners=0.000cm] at (9.625,12.625) {$*$};
            \node [font=\fontsize{8.0pt}{10.4pt}\selectfont, inner xsep=0.080cm, inner ysep=0.085cm, rounded corners=0.016cm] at (9.875,16.25) {$*^+$};
            \node [font=\fontsize{8.0pt}{10.4pt}\selectfont, inner xsep=0.080cm, inner ysep=0.085cm, rounded corners=0.000cm] at (9.825,14.375) {$*^-$};
            \draw [short] (10.875,12.75) .. controls (9.625,13.625) and (9.625,11.875) .. (10.875,12.75);
            \draw [short] (10.875,15.25) .. controls (10,16.625) and (9.375,15.625) .. (10.875,15.25);
            \draw [short] (10.875,15.25) .. controls (9.5,14.625) and (10.25,13.875) .. (10.875,15.25);
            \node [font=\fontsize{10.2pt}{13.3pt}\selectfont, inner xsep=0.080cm, inner ysep=0.085cm, rounded corners=0.010cm] at (9.025,14.25) {$\phi_4$:};
            \end{circuitikz}
            \end{center}
        \end{minipage}
		\caption{The Harmonic Morphisms giving rise to the gluing maps in (3) and (4)}
		\label{fig: gluing 3}
	\end{figure}
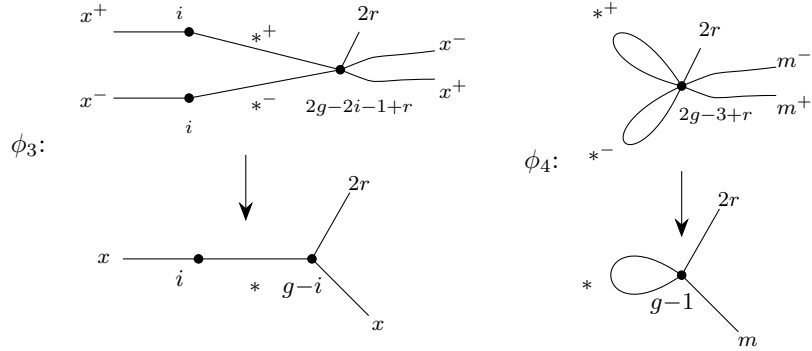
    \item Define the clutching map
    \[
    \overline\Mcal_{g-1, m + 2} \rarr \overline\Rcal_{g;m}',\quad [C, (x_j)_{j = 1}^{m + 2}] \mapsto [(C\cup_{\{x_{m + 1}, x_{m + 2}\}}C) / C, (x_j)_{j = 1}^m, (x_j^\pm)_{j = 1}^m].
    \]
    See Figure \ref{fig: gluing 5}.1 for the graph--theoretic counterpart.
    \item Finally, we define the clutching map
    \begin{align*}
    &\overline\Rcal_{g-1, 2r+2;m}' \rarr \overline\Rcal_{g, 2r;m}', \quad [\widetilde C / C, (p_i)_{i = 1}^{2r + 2}; (x_j)_{j = 1}^m, (x_j^\pm)_{j = 1}^m] \mapsto \\
    &\mapsto [\widetilde C / \{\widetilde p_{2r + 1}\sim \widetilde p_{2r + 2}\} \rarr C / \{p_{2r + 1} \sim p_{2r + 2}\}, (p_i)_{i = 1}^{2r}; (x_j^\pm)_{j = 1}^m]
    \end{align*}
    where $\widetilde p_i \in \widetilde C$ is the ramification point over $p_i \in C$. See Figure \ref{fig: gluing 5}.2.
    \begin{figure}[!ht]
    \hskip -2 em
		\begin{minipage}{0.225\textwidth}
			\begin{center}
				\begin{circuitikz}
\tikzstyle{every node}=[font=\fontsize{6.8pt}{8.9pt}\selectfont]
\node [font=\fontsize{9.1pt}{11.8pt}\selectfont, inner xsep=0.080cm, inner ysep=0.085cm, rounded corners=0.020cm] at (10.75,12.375) {$g\hspace{-2pt}-\hspace{-2pt}1$};
\begin{scope}[transparency group, opacity=1]
\node at (10.875,12.75) [circ, color={rgb,255:red,0; green,0; blue,0}] {};
\end{scope}
\draw [short] (10.875,12.75) -- (12.125,12.75);
\draw [-{Stealth[scale=1.5]}, ] (10.875,14.125) -- (10.875,13.25);
\begin{scope}[transparency group, opacity=1]
\node at (10.875,14.5) [circ, color={rgb,255:red,0; green,0; blue,0}] {};
\end{scope}
\node [font=\fontsize{8.0pt}{10.4pt}\selectfont, inner xsep=0.080cm, inner ysep=0.085cm, rounded corners=0.000cm] at (12.625,15.5) {$m^-$};
\node [font=\fontsize{6.8pt}{8.9pt}\selectfont, inner xsep=0.080cm, inner ysep=0.085cm, rounded corners=0.000cm] at (11,14.25) {$g\hspace{-2pt}-\hspace{-2pt}1$};
\draw [short] (10.875,15.375) -- (12.125,15.375);
\draw [short] (10.875,14.5) -- (12.125,14.5);
\node [font=\fontsize{8.0pt}{10.4pt}\selectfont, inner xsep=0.080cm, inner ysep=0.085cm, rounded corners=0.000cm] at (12.375,12.75) {$m$};
\node [font=\fontsize{8.0pt}{10.4pt}\selectfont, inner xsep=0.080cm, inner ysep=0.085cm, rounded corners=0.016cm] at (12.625,14.5) {$m^+$};
\node [font=\fontsize{8.0pt}{10.4pt}\selectfont, inner xsep=0.080cm, inner ysep=0.085cm, rounded corners=0.000cm] at (9.875,12.625) {$*$};
\draw [short] (10.875,12.75) .. controls (9.75,13.5) and (9.75,12) .. (10.875,12.75);
\node [font=\fontsize{8.0pt}{10.4pt}\selectfont, inner xsep=0.080cm, inner ysep=0.085cm, rounded corners=0.000cm] at (9.75,14.875) {$*^-$};
\draw [short] (10.875,14.5) .. controls (10.25,14.375) and (10.25,15.25) .. (10.875,15.375);
\draw [short] (10.875,14.5) .. controls (9.75,14.5) and (9.75,15.25) .. (10.875,15.375);
\begin{scope}[transparency group, opacity=1]
\node at (10.875,15.375) [circ, color={rgb,255:red,0; green,0; blue,0}] {};
\end{scope}
\node [font=\fontsize{8.0pt}{10.4pt}\selectfont, inner xsep=0.080cm, inner ysep=0.085cm, rounded corners=0.016cm] at (10.75,14.875) {$*^+$};
\node [font=\fontsize{6.8pt}{8.9pt}\selectfont, inner xsep=0.080cm, inner ysep=0.085cm, rounded corners=0.000cm] at (10.875,15.75) {$g\hspace{-2pt}-\hspace{-2pt}1$};
\node [font=\fontsize{10.2pt}{13.3pt}\selectfont, inner xsep=0.080cm, inner ysep=0.085cm, rounded corners=0.020cm] at (8.75,13.75) {$\phi_5$:};
\end{circuitikz}
                \end{center}
		\end{minipage}
        \qquad
        \qquad
        \qquad 
        \qquad 
        \begin{minipage}{0.225\textwidth}
			\begin{center}
            \begin{circuitikz}
\tikzstyle{every node}=[font=\fontsize{8.0pt}{10.4pt}\selectfont]
\node [font=\fontsize{9.1pt}{11.8pt}\selectfont, inner xsep=0.080cm, inner ysep=0.085cm, rounded corners=0.020cm] at (10.75,12.375) {$g\hspace{-2pt}-\hspace{-2pt}1$};
\begin{scope}[transparency group, opacity=1]
\node at (10.875,12.75) [circ, color={rgb,255:red,0; green,0; blue,0}] {};
\end{scope}
\draw [short] (10.875,12.75) -- (11.375,13.625);
\draw [short] (10.875,12.75) -- (11.625,12);
\draw [-{Stealth[scale=1.5]}, ] (10.875,14.125) -- (10.875,13.25);
\begin{scope}[transparency group, opacity=1]
\node at (10.875,14.875) [circ, color={rgb,255:red,0; green,0; blue,0}] {};
\end{scope}
\node [font=\fontsize{8.0pt}{10.4pt}\selectfont, inner xsep=0.080cm, inner ysep=0.085cm, rounded corners=0.000cm] at (12.375,15.25) {$m^-$};
\node [font=\fontsize{6.8pt}{8.9pt}\selectfont, inner xsep=0.080cm, inner ysep=0.085cm, rounded corners=0.020cm] at (11.125,14.5) {$2g\hspace{-2pt}-\hspace{-2pt}3\hspace{-2pt}+\hspace{-2pt}r$};
\draw [short] (10.875,14.875) -- (11.125,15.375);
\draw [short] (10.875,14.875) .. controls (11.5,15.125) and (11.125,15) .. (12.125,15.125);
\draw [short] (10.875,14.875) .. controls (11.5,14.625) and (11.125,14.75) .. (12.125,14.75);
\node [font=\fontsize{8.0pt}{10.4pt}\selectfont, inner xsep=0.080cm, inner ysep=0.085cm, rounded corners=0.000cm] at (11.75,11.875) {$m$};
\node [font=\fontsize{8.0pt}{10.4pt}\selectfont, inner xsep=0.080cm, inner ysep=0.085cm, rounded corners=0.016cm] at (12.375,14.625) {$m^+$};
\node [font=\fontsize{8.0pt}{10.4pt}\selectfont, inner xsep=0.080cm, inner ysep=0.085cm, rounded corners=0.000cm] at (11.5,13.75) {$2r$};
\node [font=\fontsize{8.0pt}{10.4pt}\selectfont, inner xsep=0.080cm, inner ysep=0.085cm, rounded corners=0.000cm] at (11.25,15.625) {$2r$};
\node [font=\fontsize{8.0pt}{10.4pt}\selectfont, inner xsep=0.080cm, inner ysep=0.085cm, rounded corners=0.000cm] at (10.15,13.125) {$*$};
\draw [short] (10.875,12.75) .. controls (9.5,13.625) and (9.5,11.875) .. (10.875,12.75);
\node [font=\fontsize{8.0pt}{10.4pt}\selectfont, inner xsep=0.080cm, inner ysep=0.085cm, rounded corners=0.020cm] at (10.15,12.35) {$\star$};
\begin{scope}[transparency group, opacity=1]
\node at (9.875,12.75) [circ, color={rgb,255:red,0; green,0; blue,0}] {};
\end{scope}
\draw [short] (10.875,14.875) .. controls (9.5,15.75) and (9.5,14) .. (10.875,14.875);
\begin{scope}[transparency group, opacity=1]
\node at (9.875,14.875) [circ, color={rgb,255:red,0; green,0; blue,0}] {};
\end{scope}
\node [font=\fontsize{8.0pt}{10.4pt}\selectfont, inner xsep=0.080cm, inner ysep=0.085cm, rounded corners=0.000cm] at (10.15,15.25) {$*$};
\node [font=\fontsize{8.0pt}{10.4pt}\selectfont, inner xsep=0.080cm, inner ysep=0.085cm, rounded corners=0.020cm] at (10.15,14.4) {$\star$};
\node [font=\fontsize{10.2pt}{13.3pt}\selectfont, inner xsep=0.080cm, inner ysep=0.085cm, rounded corners=0.020cm] at (8.75,13.65) {$\phi_6$:};
\end{circuitikz}
            \end{center}
        \end{minipage}
		\caption{The Harmonic Morphisms giving rise to the clutching maps in (5) and (6)}
		\label{fig: gluing 5}
	\end{figure}
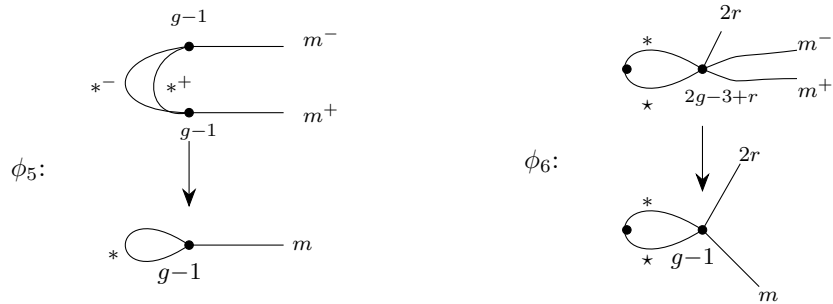
    \end{enumerate}
    See also \cite[Section 3]{hulek} for a similar combinatorial description of the boundary divisors.
\end{example}

The collection of gluing and clutching maps described in Example \ref{ex: gluing and clutching maps} allows us to give an intrinsic definition of the tautological ring of $\rgrmbar'$.

\begin{definition}
    \begin{enumerate}
        \item The \textit{system of tautological rings $\{\R^\bullet(\rgrmbar')\}_{g, r, m}$} is the smallest system of $\QQ$--subalgebras of $\{\CH^\bullet(\rgrmbar')\}_{g, r, m}$ which contains the fundamental classes $[\rgrmbar']$ and is stable under pushforward by the gluing and clutching maps from Example \ref{ex: gluing and clutching maps}, as well as pushforward under forgetful maps $\overline\Rcal_{g, 2r;m + 1}' \rarr \rgrmbar'$. The tautological ring of $\rgrmbar$ is the image of $\R^\bullet(\rgrmbar')$ inside $\CH^\bullet(\rgrmbar)$ under the natural pushforward map $\CH^\bullet(\rgrmbar') \rarr \CH^\bullet(\rgrmbar)$.
        \item The \textit{cohomological tautological ring} or $\rgrmbar'$ (respectively $\rgrmbar$) is the image of the tautological ring $\R^\bullet(\rgrmbar')$ inside $\H^\bullet(\rgrmbar')$ under the cycle class map $\CH^\bullet \xrightarrow[]{\mathrm{cl}} \H^\bullet$.
    \end{enumerate}
\end{definition}

As for $\ol{\mc M}_{g,n}$, we have an obvious notion of specialisation for Prym structures on $G$. 

\begin{definition}\label{def: Specialisation Prym}
    Given a Prym structure $\phi:G'\ra G$ we will say that a Prym structure $\psi:A'\ra A$ is a \textit{specialisation} of $\phi$ if one may obtain it from $\phi$ by replacing each vertex and its preimages by Prym structures between graphs of the same genus and valence.
\end{definition}

\begin{remark}
    Notice that, when $\#\phi^{-1}(v)=2$, the only possible specialisations of $\phi$ at $v$ are those given by replacing $v$ by a graph of the same genus and valence and by replacing its preimages by \textit{the same} graph. 
\end{remark}

\begin{remark}
    By definition, points of $\ol{\mc R}'_\phi$ correspond to Prym curves together with a $\phi$--structure on the Prym structure on their dual graphs. 
\end{remark}

\begin{definition}
    A Prym structure $\phi:G'\ra G$ is said to have a \textit{$(\phi_1,\phi_2)$--structure} if it carries both a $\phi_1$ and a $\phi_2$--structure. Furthermore, we will say that $\phi$ is a \textit{generic} $(\phi_1,\phi_2)$--Prym structure if for each $h\in H(G)$, the restriction $\phi^{-1}(h)\ra h$ is given by $\phi_1$ or $\phi_2$.
\end{definition}

\begin{lemma}
    Let $\phi_1$ and $\phi_2$ be Prym structures on graphs $G_1$ and $G_2$. Let $\ol{\mc R}'_{\phi_1,\phi_2}$ be the disjoint union of $\ol{\mc R}'_{\phi}$ where $\phi$ is a generic $(\phi_1,\phi_2)$--Prym structure. Then, the following diagram is Cartesian.
    \[
    \begin{tikzcd}
        \ol{\mc R}'_{\phi_1,\phi_2} \arrow[r, ] \arrow[d,]  & \ol{\mc R}'_{\phi_1}\arrow[d, "\chi_{\phi_1}"]\\
        \ol{\mc R}'_{\phi_2} \arrow[r, "\chi_{\phi_2}"]  & \ol{\mc R}'_{g;m} \arrow[ul, phantom, very near end, "\ulcorner"]
    \end{tikzcd}
    \]
\end{lemma}
\begin{proof}
    Mutatis mutandis, the proof proceeds as in \cite[Appendix A, Proposition 9]{pandhairpande2003nontaut}.
\end{proof}
Notice now that we may use the local description of the forgetful map, $u:\ol{\mc R}'_{g;m}\ra \ol{\mc M}_{g,m}$, to easily compute the normal bundle of $\chi_\phi:\ol{\mc R}'_{\phi}\ra\ol{\mc R}'_{g;m}$.

\begin{lemma}\label{lem: normal bundle}
    Let $\phi:G'\ra G$ be a Prym structure on a graph $G$. Then the normal bundle to $\chi_\phi: \ol{\mc R}'_\phi\ra \rgmbar'$ is:
    \[\mc N_\phi\cong \bigoplus_{e=\{h,\ol h\}\in E^{\mathrm{ram}}(G)^c}u^*\left(\mathbb L_h^\vee\otimes \mathbb L_{\ol h}^\vee\right)\quad \oplus \bigoplus_{e=\{h,\ol h\}\in E^{\mathrm{ram}}(G)}M_e\]
    where $M_e^{\otimes 2}\cong u^*\left(\mathbb L_h^\vee\otimes \mathbb L_{\ol h}^\vee\right)$ and $\mathbb L_{h}$ is the line bundle given by the cotangent lines at $h$.
\end{lemma}
\begin{proof}
    Let $(\tau_i)_{i = 1}^{3g - 3 + m}$ denote the local deformation parameters at $[\widetilde C\ra C; (x_j)_{j = 1}^m]\in \Delta_\phi$ and $(t_i)_{i = 1}^{3g - 3+ m}$ be the ones at $\st(C) \in \overline{\Mcal}_{g, m}$. Assume that $(t_i)_{i = 1}^r$ are the deformation parameters of the exceptional nodes of $\mathrm{st}(C)$, $(\tau_i)_{i = 1}^r$ those of the corresponding exceptional components of $C$, $(t_j)_{j=r+1}^s$ those of the non--exceptional nodes of $\mathrm{st}(C)$ and $(\tau_j)_{j=r+1}^s$ those of the corresponding nodes of $C$ so that, locally around $[\widetilde C\ra C; (x_j)_{j = 1}^m]$, we have that $\Delta_\phi$ is given by $\tau_i=0$ for $i\leq s$ and $u(\Delta_\phi)=\Delta_{\st(G)}$ by $t_i=0$ for $i\leq s$. By the deformation theory of ramified covers, the local expression of $u$ around $[\widetilde C/C; (x_j)_{j = 1}^m]$ is given by:
    \[
    \CC^{3g-3+m}_\tau\ra \CC^{3g-3+m}_{t}, \qquad t_i\mapsto\begin{cases}
        \tau_i^2 \quad \text{if }1 \le i \le r \\
        \tau_i \quad \text{if }r + 1 \le i \le 3g - 3 + m.
    \end{cases}
    \]
    Calling $\Delta_i$ the boundary divisor of $\ol{\mc M}_{g,m}$ locally given by $t_i=0$ and by $\Delta'_i$ the one of $\rgmbar$ given by $\tau_i=0$, we then have $u^*\mc N_{\Delta_i}\cong \mc N_{\Delta'_i}^{\otimes 2}$ for $i\leq r$ and $u^*\mc N_{\Delta_i}\cong \mc N_{\Delta'_i}$ for $r+1\leq i\leq s$.
    \vskip 0.5 em 
    The result now follows by recalling that, calling $h$ and $\ol h$ the two half--edges of $\st(G)$ corresponding to the node of $\st(C)$ having $t_i$ as deformation parameter, the normal bundle of the boundary divisor $\Delta_i$ is $\mathbb L_{h}^\vee\otimes \mathbb L_{\ol h}^\vee$.
\end{proof}

Now, let $\phi_1:G'_1\ra G_1$, $\phi_2:G'_2\ra G_2$ be Prym structures on graphs $G_1$ and $G_2$ and let $\phi$ be a generic $(\phi_1,\phi_2)$--Prym structure.
Denote by $u:\ol{\mc R}_\phi'\ra\ol{\mc M}_{\mathrm{st}(G_1)}$ the forgetful morphism and, for each vertex $v$ of $\st(G_1)$, denote by $p_v:\ol{\mc M}_{\mathrm{st}(G_1)}\ra \ol{\mc M}_{g(v),n(v)}$ the projection map. 
\vskip 0.5 em
Let $q_{1,\phi}(\underline{\psi_h}) \in \CH^\bullet(\ol{\mc R}_{\phi_1})$ be the following polynomial in the $\psi$ classes:
\[q_{1,\phi}(\underline{\psi_h}):=\prod_{e=\{h,\ol{h}\}}-u^*\left(p_{a(h)}^*\left(\psi_h\right)+p_{a(\ol{h})}^*\left(\psi_{\ol{h}}\right)\right)\] where the product runs over all edges $e$ of $G$ that correspond to both an edge of $G_1$ and of $G_2$, and that are non-exceptional in at least one of the two.
Let also $q_{2,\phi}(\underline{\psi_h})\in \CH^\bullet(\ol{\mc R}_{\phi_1})$ be the following polynomial in the $\psi$ classes:
\[q_{2,\phi}(\underline{\psi_h}):=\prod_{e=\{h,\ol{h}\}}-\frac{1}{2}u^*\left(p_{a(h)}^*\left(\psi_h\right)+p_{a(\ol{h})}^*\left(\psi_{\ol{h}}\right)\right)\] where the product runs over all edges $e$ of $G$ that correspond to both an exceptional edge of $G_1$ and of $G_2$. The following holds:
\begin{corollary}\label{cor: pullback boundary classes}
    With the above notations,
    \[\chi_{\phi_1}^*\left(\chi_{\phi_2,*}\left[\ol{\mc R}'_{\phi_2}\right]\right)=\sum_{\phi} \chi_{\phi,\phi_1,*}\left(q_{1,\phi}(\underline{\psi_h})\cdot q_{2,\phi}(\underline{\psi_h})\right)
    \]
    where the sum runs over all generic $(\phi_1,\phi_2)$--Prym structures.
\end{corollary}
\begin{proof}
    The corollary follows from Lemma \ref{lem: normal bundle} and the Excess Intersection Formula.
\end{proof}
\begin{proposition}\label{prop: pullback by chi admit tautological kunneth decomp}
    Let $\phi$ be a Prym--structure and let $\chi_\phi : \ol{\mc R}_{\phi}' \rarr \overline\Rcal_{g;m}'$ be the corresponding glueing morphism. Then all classes in $\chi_\phi^*(\R\H^\bullet(\rgmbar'))$ admit a tautological Künneth decomposition.
\end{proposition}
\begin{proof}
    Using Corollary \ref{cor: pullback boundary classes}, the proof proceeds exactly as in \cite[Appendix A]{pandhairpande2003nontaut}. 
\end{proof}
\vskip 0.5em

\section{The Cohomology of \texorpdfstring{$\ol{\mc R}_{1;m}$}{}}\label{section: cohomology R1m}

\subsection{The moduli spaces \texorpdfstring{$\overline\Rcal_{1;m}$}{}}

We now gather some facts on the Chow ring $\CH^\bullet(\overline\Rcal_{1;m})$. The statements in this subsection serve as a starting point to our results on the Chow ring of $\CH^\bullet(\rgmbar)$ via the glueing maps described in Example \ref{ex: gluing and clutching maps}. 

\begin{notation}
    Let $\Gamma \subset \SL(2, \ZZ)$ be a congruence subgroup. We denote by $\Scal_k(\Gamma)$ the space of cusp forms of level $\Gamma$ and weight $k$. We also set $\Ecal_k(\Gamma)$ for the space of Eisenstein series of level $\Gamma$ and weight $k$.
\end{notation}

In order to better understand the cohomology $\H^\bullet(\overline\Rcal_{1;m})$, we follow standard methods which have already proven their usefulness for $\overline\Mcal_{1, n}$. The following result is foundational.
\begin{theorem}[Eichler--Shimura]
    Let $\Gamma \subset \SL(2, \ZZ)$ be a congruence subgroup. There exists a Hecke--equivariant isomorphism
    \[
    \overline{\Scal_k(\Gamma)} \oplus \Scal_k(\Gamma) \oplus \Ecal_k(\Gamma) \xrightarrow[]{\sim} \H^1(\Gamma, \Sym^{k -2}(\CC^2)).
    \]
    Here, $\Sym^{k -2}(\CC^2)$ is a $\Gamma$--module via the inclusion $\Gamma \subset \GL(2, \CC)$.
\end{theorem}

Consider now the congruence subgroup
\[
\Gamma_1(2):= \left\{
\begin{pmatrix}
    a & b \\
    c & d
\end{pmatrix} \in \SL(2, \ZZ)\colon
\begin{pmatrix}
    a & b \\
    c & d
\end{pmatrix} \equiv \begin{pmatrix}
    1 & * \\
    0 & 1
\end{pmatrix} \Mod{2}
\right\},
\]
and notice that $X_1(2) := \HH/\Gamma_1(2) \cong \Rcal_{1;1}$. Motivated by the Eichler--Shimura Isomorphism, we now state a basic lemma on cusp forms of level $\Gamma_1(2)$. 
\begin{lemma}\label{lemma: first nonzero cusp form for Gamma_1(2)}
    Let $k \in \ZZ_{\ge 0}$ be an even number. The dimension of the space of weight $k$ cusp forms of level $\Gamma_1(2)$ is given by the formula:
    \[
    \dim \Scal_k(\Gamma_1(2)) = \begin{cases}
        \left\lfloor \frac{k}{4} \right\rfloor- 1 &\text{for }k \ge 8 \\
        0 &\text{for $k \le 6$.}
    \end{cases}
    \]
    In particular, the first non--zero cusp form of even weight and level $\Gamma_1(2)$ appears in weight 8.
\end{lemma}
\begin{proof}
    The Lemma follows by combining \cite[Theorem 3.5.1]{diamond2005modularForms} and \cite[Exercise 3.1.5]{diamond2005modularForms}, which yield $\varepsilon_2 = 1$, $\varepsilon_3 = 0$ and $\varepsilon_\infty = 2$. Here, $\varepsilon_k$ stands for the number of elliptic points of period $k$ for $k = 2, 3$, and $\varepsilon_\infty$ is the number of cusps. We also observe that $\Rcal_{1;1} \cong \HH/\Gamma_1(2)$ is rational so the genus $g(\HH/\Gamma_1(2)) = 0$. 
\end{proof}

\begin{notation}
    For the rest of this subsection, we let $\nu : \Ecal \rarr \Rcal_{1;1}$ be the universal elliptic curve, $\VV := \R^1\nu_*\underline{\QQ}$, and $\VV_k := \Sym^k\,\VV$.
\end{notation}
\begin{lemma}\label{lemma: odd coh vanishing for loc system V_k on R 1,1}
    For $k$ odd, we have $\H^\bullet(\Rcal_{1;1}, \VV_k) = 0$.
\end{lemma}
\begin{proof}
    We need to show that $\H^\bullet(\Gamma_1(2), \Sym^k(\CC^2)) = 0$. To this end, recall that, for a group $G$, and functors $\H^q(G, -)$ are defined as the left derived functors of
    \[
    (-)^G : G\mathrm{-Mod} \rarr \mathrm{Ab}, \quad M \mapsto M^G:= \{m \in M \colon g\cdot m = m \text{ for all } g \in G\}.
    \]
    On one hand, since $-\Id_2 \in \Gamma_1(2)$, it acts as the identity on $\H^\bullet(\Gamma_1(2), \Sym^k(\CC^2))$. On the other, since $k$ is odd, $-\Id_2$ acts on $\Sym^k(\CC^2)$ by $v_1\dots v_k \mapsto -v_1\dots v_k$. It follows that $\alpha = - \alpha$ for all $\alpha \in \H^\bullet(\Gamma_1(2), \Sym^k(\CC^2))$ implying the desired vanishing.
\end{proof}

Let $\pi : \Rcal_{1;m}^\rt \rarr \Rcal_{1;1}$ be the forgetful map from the space $\Rcal_{1;m}^\rt$ of elliptic Prym curves with rational tails. Recall that a genus $g$ stable curve $C$ is said to have rational tails if it has component of geometric genus $g$. We see that $\Rcal_{1;m}^\rt$ fits into the following cartesian diagram:
\[\begin{tikzcd}
	{\Rcal_{1;m}^\rt} & {\Mcal_{1,m}^\rt} \\
	{\Rcal_{1;1}} & {\Mcal_{1, 1}}
	\arrow[from=1-1, to=1-2]
	\arrow["\pi"', from=1-1, to=2-1]
	\arrow["\ulcorner"{anchor=center, pos=0.125}, draw=none, from=1-1, to=2-2]
	\arrow["\pi'", from=1-2, to=2-2]
	\arrow["u", from=2-1, to=2-2]
\end{tikzcd}\]

We next set to prove the following proposition.
\begin{proposition}\label{prop: key prop for even and odd coh of R1,m}
    \hspace{-50pt}
    \begin{enumerate}[label=$\roman*)$]
        \item Let $F_m$ be the fibre of $\pi : \Rcal_{1;m}^\rt \rarr \Rcal_{1;1}$. If $k$ is even, then
        \[
        \H^k(\Rcal^\rt_{1;m}) \cong \H^0(\Rcal_{1;1}, \R^k\pi_*\underline{\QQ}) = \H^k(F_m)^{\Gamma_1(2)}.
        \]
        \item If $k$ is odd, then
        \[
        \H^k(\Rcal^\rt_{1;m}) \cong \H^1(\Rcal_{1;1}, \R^{k-1}\pi_*\underline{\QQ}).
        \]
    \end{enumerate}
\end{proposition}
\begin{proof}
    By Cohomology and Base Change, we have canonical isomorphisms of local systems
\begin{align}\label{eq: coh base change to R1m}
\R^q\pi_*\underline{\QQ} \cong u^*\R^q\pi'_*\underline{\QQ}.
\end{align}
The classification of irreducible representations of $\SL(2, \ZZ)$ together with (\ref{eq: coh base change to R1m}) imply that $\R^p\pi_*\underline{\QQ}$ splits as a direct sum into irreducible local systems
\begin{align*}
\R^q\pi_*\underline{\QQ} = \bigoplus_j \VV_{k_j}(-n_j)
\end{align*}
where $k_j + 2n_j = q$ for all $j$. In particular, if $q$ is odd (respectively even), $\R^q\pi_*\underline{\QQ}$ contains only local systems $\VV_k$ with $k$ odd (respectively even). Based on Lemma \ref{lemma: odd coh vanishing for loc system V_k on R 1,1} and since $\Rcal_{1;1}$ has dimension 1, the Leray spectral sequence for $\pi : \Rcal_{1;m}^\rt \rarr \Rcal_{1;1}$ degenerates on the second page. The proposition now follows.
\end{proof}

We have thus arrived at the main result of this section.
\begin{theorem}\label{thm: key basic result on R1m}
    \hspace{-50pt}
    \begin{enumerate}[label=$\roman*)$]
        \item The Chow and cohomology rings of $\overline\Rcal_{1;m}$ are tautological for $m \le 6$.
        \item For all $m \ge 7$ odd we have that $\H^{m, 0}(\overline{\Rcal}_{1;m}) \ne 0$.
        \item All the even cohomology of $\overline{\Rcal}_{1;m}$ is tautological.
    \end{enumerate}
\end{theorem}
\begin{proof}
    Part $i)$ follows from \cite[Chapter 4]{krug2012thesisSpinPrym}.

    \vskip 0.5em
    Part $ii)$ follows from the Eichler--Shimura isomorphism for the congruence subgroup $\Gamma_1(2)$ together with Lemma \ref{lemma: first nonzero cusp form for Gamma_1(2)} and Proposition \ref{prop: key prop for even and odd coh of R1,m}. Indeed, by the usual dictionary between representations and local systems, there exists a canonical isomorphism $\H^1(\Rcal_{1;1}, \VV_k) \cong \H^1(\Gamma_1(2), \Sym^k(\CC^2))$. The map $\pi : \Rcal_{1;m}^\rt \rarr \Rcal_{1;1}$ factors through the $(m-1)$--fold product of the universal elliptic curve $\Rcal_{1;m}^\rt \xrightarrow[]{}\Ecal^{m - 1} \xrightarrow[]{\nu^{m -1}} \Rcal_{1;1}$. Since $\Rcal_{1;1}$ is affine, we have an isomorphism $\H^\bullet(\Ecal^{m-1}) \cong \H^\bullet(E^{m -1})$ where $[E] \in \Rcal_{1;1}$ is an elliptic curve. By Künneth's formula, the local system $\R^k\nu^{m - 1}_*\underline{\QQ}$ can be written in terms of tensor powers $\VV^{\otimes l}$. Now, note that the tensor powers $\VV^{\otimes l}$ admit decompositions in terms of symmetric powers and that there appears only one copy of $\VV_k$ in the decomposition of $\VV^{\otimes k}$. This implies that $\H^1(\Rcal_{1;1}, \VV_{k -1})$ appears as a direct summand in $\H^1(\Rcal_{1;1}, \R^{k -1}\pi_*\underline{\QQ})$. By Lemma \ref{lemma: first nonzero cusp form for Gamma_1(2)}, the first non--zero cusp form of level $\Gamma_1(2)$ and even weight appears in weight 8. The Eichler--Shimura Isomorphism finishes the proof of (2).

    \vskip 0.5em
    To address $iii)$, we closely follow \cite{petersen2014tautGenus1}. Direct computation gives the following claim on the $\Gamma_1(2)$--invariant part of $\H^1(E)^{\otimes 2}$.
    \begin{claim}
        The $\Gamma_1(2)$--invariant part of $\H^1(E)^{\otimes 2}$ is
        \[
        \left[\H^1(E)^{\otimes 2}\right]^{\Gamma_1(2)} = \wedge^2\,\H^1(E).
        \]
    \end{claim}

    \vskip 0.5em
    From the Claim we conclude that, for $k$ even, the $\Gamma_1(2)$--invariant part of $\H^1(E)^{\otimes k}$ is generated by $\left(\wedge^2\, \H^1(E)\right)^{\otimes k/2}$ together with translates under the $\Sfrak_{k}$--action.

    \vskip 0.5em
    We also point out the important basic fact that the 1--dimensional subspace $\wedge^2\, \H^1(E) $ of $ \H^1(E)^{\otimes 2}$ is generated by the restriction of the class of the diagonal in $\H^2(E \times E)$. In conclusion, the invariant part $\H^\bullet(E^{m - 1})^{\Gamma_1(2)}$ is generated by classes of diagonals and generators of $\H^2(E^{m - 1})$. Usual methods which parallel \cite{petersen2014tautGenus1} allow us to conclude that $\H^\bullet(E^m / E)$ is generated by classes of diagonals, where $E$ acts diagonally on $E^m$.

    \vskip 0.5em
    Let now $\text{FM}(E, n)$ denote the Fulton--MacPherson compactification of $E^n$ so that there exists an isomorphism
    \[
    F_m \cong \text{FM}(E, m)/E.
    \]
    The cohomology of $\text{FM}(E, m)$ is generated over the cohomology of $E^m$ by classes of exceptional divisors. Since we are only blowing up $\Gamma_1(2)$--invariant loci, the action of $\Gamma_1(2)$ extends from $\H^\bullet(E^m)$ to $\H^\bullet(\text{FM}(E, m))$ in the obvious way, which is to say, it is trivial on each exceptional component.

    \vskip 0.5em
    Thus, we have shown that the $\Gamma_1(2)$--invariants of $\H^\bullet(F_m)$ are generated as an algebra by the $\Gamma_1(2)$--invariants of $\H^\bullet(E^{m -1})$. Applying Proposition \ref{prop: key prop for even and odd coh of R1,m}, we find that $\H^{2\bullet}(\Rcal_{1;m}^\rt)$ is generated by boundary cycles. To finish the proof, recall that we the long exact sequence arising from excision:
    \[
    \H^{k - 2}(\overline\Rcal_{1;m} \setminus \Rcal_{1;m}^{\rt}) \rarr \H^k(\overline\Rcal_{1;m}) \rarr \H^k(\Rcal_{1;m}^\rt).
    \]
    The components of the complement $\overline{\Rcal}_{1;m} \setminus\Rcal_{1;m}^\rt$ are given by moduli spaces of genus 0 Prym curves, and the cohomology of these spaces is generated by boundary classes, concluding the proof of the Theorem.
\end{proof} 

Using the gluing maps that describe the boundary of $\overline{\Rcal}_{g;m}$ from Example \ref{ex: gluing and clutching maps}, we deduce the following corollary of Theorem \ref{thm: key basic result on R1m} and Proposition \ref{prop: pullback by chi admit tautological kunneth decomp}.

\begin{corollary}\label{cor: classes on bd R1 7 taut}
    \hspace{-50 pt}\begin{enumerate}[label=$\roman*)$]
        \item Every algebraic class on $\overline\Rcal_{1;7}\times \ol{\Rcal}_{1;7}$ of codimension $7$ supported on $\partial(\Rcal_{1;7}\times \Rcal_{1;7})$ has tautological K\"unneth decomposition.
        \item Every algebraic class on $\overline\Rcal_{1;7}\times \ol{\Rcal}_{1;7}$ of codimension less then $7$ has tautological K\"unneth decomposition.
    \end{enumerate}
\end{corollary}

The following Lemma, which was pointed to us by Samir Canning, is proven identically to \cite[Fact 2.1]{canningNonTautCycles}.
\begin{lemma}\label{lemma: non-vanishing hol forms on interior}
    Every non--zero element in $\H^{k, 0}(\rgmbar)$ has non--zero image under the restriction map $\H^k(\rgmbar) \rarr \H^k(\rgm)$.
\end{lemma}

\section{Proof of the Main Theorems}\label{section: proof of thm 11}

In this section we prove Theorem \ref{thm: classes RBbar are non-taut} and Theorem \ref{thm: main non--tautology theorem}.

\subsection{A key lemma}

We next focus on the glueing morphism $\chi : \overline\Rcal_{1;k}' \times \overline\Rcal_{1;k}' \rarr \overline\Rcal_{g;2m}'$. 


\vskip 0.5em
\begin{lemma}\label{lemma: pullback of biell cmp in Rgm is a multiple of the diagonal}
    Let $g \ge 2$, $g + 2m \ge 8$ and $k : = g - 1 + m$. The pullback under the glueing map $\overline\Rcal_{1;k}' \times \overline\Rcal_{1;k}' \xrightarrow[]{\chi} \overline\Rcal_{g; 2m}' \rarr \overline\Rcal_{g;2m}$ of the class of $\overline{\Rcal\Bcal}_{g, 0, 2m}^0 \subset \rgmbar$, restricted to the interior $\Rcal_{1;k}' \times \Rcal_{1;k}'$, is a non--zero multiple of the class of the diagonal $\Delta \subset \Rcal_{1;k}' \times \Rcal_{1;k}'$, i.e. $\chi_{|_{\Rcal_{1;k}' \times \Rcal_{1;k}'}}^*\left[\overline{\Rcal\Bcal}_{g, 0, 2m}^0\right] = \alpha \cdot [\Delta]$ for some $\alpha \in \QQ_{> 0}$.
\end{lemma}
\begin{proof}
Define the moduli stack $\overline{\Rcal\Adm}(g, 1)_{2m}^0$ by the following cartesian square:
\[\begin{tikzcd}
	{\overline{\Rcal\Adm}(g, 1)_{2m}^0} & {\overline{\Adm}(g, 1)_{2m}} \\
	{\overline{\Rcal\Bcal}_{g, 0, 2m}^0} & {\Bcal_{g, 0, 2m},}
	\arrow[from=1-1, to=1-2]
	\arrow[from=1-1, to=2-1]
	\arrow["\ulcorner"{anchor=center, pos=0.125}, draw=none, from=1-1, to=2-2]
	\arrow[from=1-2, to=2-2]
	\arrow[from=2-1, to=2-2]
\end{tikzcd}\]
and consider the following diagram:
\begin{align}\label{eq: key commutative diagram with admissible covers}
\begin{tikzcd}[ampersand replacement=\&]
	{\Rcal_{1;k}'} \&\& \\
	\& F \& {\overline{\Rcal\Adm}(g, 1)_{2m}^0} \\
	\& {\Rcal_{1;k}'\times \Rcal_{1;k}'} \& {\overline{\Rcal}_{g;2m}.}
	\arrow["\zeta"', dashed, from=1-1, to=2-2]
	\arrow["\theta", curve={height=-12pt}, from=1-1, to=2-3]
	\arrow["\delta"', curve={height=12pt}, from=1-1, to=3-2]
	\arrow[from=2-2, to=2-3]
	\arrow["{\widetilde \phi}"', from=2-2, to=3-2]
	\arrow["\ulcorner"{anchor=center, pos=0.125}, draw=none, from=2-2, to=3-3]
	\arrow["\phi", from=2-3, to=3-3]
	\arrow["\chi", from=3-2, to=3-3]
\end{tikzcd}
\end{align}
The map $\phi$ is given by stabilisation. The map $\theta$ joins two copies of $E$ by rational components at the first $g - 1$ points $x_i$ and two copies of $\widetilde E$ by rational components at the points $x_i^\pm$ (a total of $2g - 2$ rational components upstairs). See Figure \ref{fig: def theta in a picture}.

\begin{figure}[!ht]
\centering
\resizebox{0.6\textwidth}{!}{%
\begin{circuitikz}
\tikzstyle{every node}=[font=\fontsize{10.8pt}{14.1pt}\selectfont]

\draw [short] (6.125,5.125) .. controls (11,6.5) and (11,6.5) .. (15.5,5.125)node[pos=1, fill={rgb,255:red,255; green,255; blue,255}, fill opacity=1, text opacity=1, inner xsep=0.080cm, inner ysep=0.085cm, rounded corners=0.020cm]{$E$};
\draw [short] (6,8.25) .. controls (10.875,9.625) and (10.875,9.625) .. (15.375,8.25)node[pos=1, fill={rgb,255:red,255; green,255; blue,255}, fill opacity=1, text opacity=1, inner xsep=0.080cm, inner ysep=0.085cm, rounded corners=0.020cm]{$E$};
\draw [short] (6.125,7.125) .. controls (11,8.5) and (11,8.5) .. (15.5,7.125)node[pos=1, fill={rgb,255:red,255; green,255; blue,255}, fill opacity=1, text opacity=1, inner xsep=0.080cm, inner ysep=0.085cm, rounded corners=0.020cm]{$E$};
\draw [short] (6.5,5.875) -- (6.5,4)node[pos=0, fill={rgb,255:red,255; green,255; blue,255}, fill opacity=1, text opacity=1, inner xsep=0.080cm, inner ysep=0.085cm, rounded corners=0.020cm]{$\PP^1$};
\node at (6.5,4.375) [circ, color={rgb,255:red,0; green,0; blue,0}] {};
\node at (6.5,4.75) [circ, color={rgb,255:red,0; green,0; blue,0}] {};
\draw [short] (8,6.25) -- (8,4.375)node[pos=0, fill={rgb,255:red,255; green,255; blue,255}, fill opacity=1, text opacity=1, inner xsep=0.080cm, inner ysep=0.085cm, rounded corners=0.020cm]{$\PP^1$};
\node at (8,4.75) [circ, color={rgb,255:red,0; green,0; blue,0}] {};
\node at (8,5.125) [circ, color={rgb,255:red,0; green,0; blue,0}] {};
\node at (14.45,5.45) [circ, color={rgb,255:red,0; green,0; blue,0}] {};
\node at (14.45,7.45) [circ, color={rgb,255:red,0; green,0; blue,0}] {};
\node at (14.45,8.525) [circ, color={rgb,255:red,0; green,0; blue,0}] {};
\node at (13.6,5.7) [circ, color={rgb,255:red,0; green,0; blue,0}] {};
\node at (13.6,8.775) [circ, color={rgb,255:red,0; green,0; blue,0}] {};
\node at (13.6,7.7) [circ, color={rgb,255:red,0; green,0; blue,0}] {};
\node at (12.75,7.915) [circ, color={rgb,255:red,0; green,0; blue,0}] {};
\node at (12.75,9) [circ, color={rgb,255:red,0; green,0; blue,0}] {};
\node at (12.75,5.91) [circ, color={rgb,255:red,0; green,0; blue,0}] {};
\draw [short] (6.625,9) -- (6.625,6.875)node[pos=0, fill={rgb,255:red,255; green,255; blue,255}, fill opacity=1, text opacity=1, inner xsep=0.080cm, inner ysep=0.085cm, rounded corners=0.020cm]{$\PP^1$};
\node at (6.625,7.625) [circ, color={rgb,255:red,0; green,0; blue,0}] {};
\node at (6.625,8) [circ, color={rgb,255:red,0; green,0; blue,0}] {};
\draw [short] (8.125,9.5) -- (8.125,7.375)node[pos=0, fill={rgb,255:red,255; green,255; blue,255}, fill opacity=1, text opacity=1, inner xsep=0.080cm, inner ysep=0.085cm, rounded corners=0.020cm]{$\PP^1$};
\node at (8.125,8.04) [circ, color={rgb,255:red,0; green,0; blue,0}] {};
\node at (8.125,8.4) [circ, color={rgb,255:red,0; green,0; blue,0}] {};

\draw [short] (5.75,11.25) .. controls (10.625,12.625) and (10.625,12.625) .. (15.125,11.25)node[pos=1, fill={rgb,255:red,255; green,255; blue,255}, fill opacity=1, text opacity=1, inner xsep=0.080cm, inner ysep=0.085cm, rounded corners=0.020cm]{$\widetilde E$};
\draw [short] (5.875,10.125) .. controls (10.75,11.5) and (10.75,11.5) .. (15.25,10.125)node[pos=1, fill={rgb,255:red,255; green,255; blue,255}, fill opacity=1, text opacity=1, inner xsep=0.080cm, inner ysep=0.085cm, rounded corners=0.020cm]{$\widetilde E$};
\node at (14.6,10.325) [circ, color={rgb,255:red,0; green,0; blue,0}] {};
\node at (14.3,11.5) [circ, color={rgb,255:red,0; green,0; blue,0}] {};
\node at (14.6,11.4) [circ, color={rgb,255:red,0; green,0; blue,0}] {};
\node at (14.3,10.4) [circ, color={rgb,255:red,0; green,0; blue,0}] {};
\node at (13.75,11.665) [circ, color={rgb,255:red,0; green,0; blue,0}] {};
\node at (13.45,11.75) [circ, color={rgb,255:red,0; green,0; blue,0}] {};
\node at (13.75,10.575) [circ, color={rgb,255:red,0; green,0; blue,0}] {};
\node at (13.45,10.675) [circ, color={rgb,255:red,0; green,0; blue,0}] {};
\node at (12.9,11.9) [circ, color={rgb,255:red,0; green,0; blue,0}] {};
\node at (12.6,11.975) [circ, color={rgb,255:red,0; green,0; blue,0}] {};
\node at (12.6,10.9) [circ, color={rgb,255:red,0; green,0; blue,0}] {};
\node at (12.9,10.81) [circ, color={rgb,255:red,0; green,0; blue,0}] {};
\draw [short] (6.375,12) -- (6.375,9.875)node[pos=0, fill={rgb,255:red,255; green,255; blue,255}, fill opacity=1, text opacity=1, inner xsep=0.080cm, inner ysep=0.085cm, rounded corners=0.020cm]{$\PP^1$};
\node at (6.375,10.63) [circ, color={rgb,255:red,0; green,0; blue,0}] {};
\node at (6.375,11) [circ, color={rgb,255:red,0; green,0; blue,0}] {};
\draw [short] (8,12.5) -- (8,10.375)node[pos=0, fill={rgb,255:red,255; green,255; blue,255}, fill opacity=1, text opacity=1, inner xsep=0.080cm, inner ysep=0.085cm, rounded corners=0.020cm]{$\PP^1$};
\node at (8,11.125) [circ, color={rgb,255:red,0; green,0; blue,0}] {};
\node at (8,11.5) [circ, color={rgb,255:red,0; green,0; blue,0}] {};
\node [font=\fontsize{10.8pt}{14.1pt}\selectfont, fill={rgb,255:red,255; green,255; blue,255}, fill opacity=1, text opacity=1, inner xsep=0.080cm, inner ysep=0.085cm, rounded corners=0.020cm] at (9.5,5) {branch points};
\node [font=\fontsize{10.8pt}{14.1pt}\selectfont, fill={rgb,255:red,255; green,255; blue,255}, fill opacity=1, text opacity=1, inner xsep=0.080cm, inner ysep=0.085cm, rounded corners=0.020cm] at (13.5,5) {markings};
\draw [-{Stealth[scale=1.5]}, ] (10.5,10.75) -- (10.5,9.75);
\draw [-{Stealth[scale=1.5]}, ] (10.5,7.625) -- (10.5,6.625);
\draw [short] (8.5,12.625) -- (8.5,10.5)node[pos=0, fill={rgb,255:red,255; green,255; blue,255}, fill opacity=1, text opacity=1, inner xsep=0.080cm, inner ysep=0.085cm, rounded corners=0.020cm]{$\PP^1$};
\node at (8.5,11.25) [circ, color={rgb,255:red,0; green,0; blue,0}] {};
\node at (8.5,11.625) [circ, color={rgb,255:red,0; green,0; blue,0}] {};
\draw [short] (6.875,12.125) -- (6.875,10)node[pos=0, fill={rgb,255:red,255; green,255; blue,255}, fill opacity=1, text opacity=1, inner xsep=0.080cm, inner ysep=0.085cm, rounded corners=0.020cm]{$\PP^1$};
\node at (6.875,10.750) [circ, color={rgb,255:red,0; green,0; blue,0}] {};
\node at (6.875,11.125) [circ, color={rgb,255:red,0; green,0; blue,0}] {};
\end{circuitikz}
}%
\caption{The image of $[\widetilde E, E]$ under $\theta$}
\label{fig: def theta in a picture}
\end{figure}
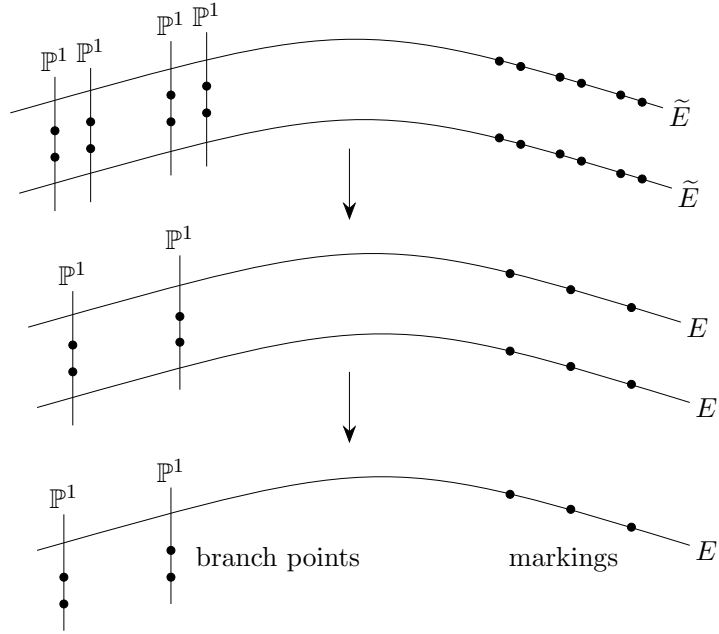

\vskip 0.5em
We point out that $\theta$ is well defined. Indeed, we see in Figure \ref{fig: def theta in a picture} that, if we forget the marked points on the rational components and stabilise, we obtain curves $[\widetilde E \cup_{\{x_i^\pm\}} \widetilde E \rarr E \cup_{\{x_i\}} E \rarr E]$. Notice also that $\widetilde E$ possesses the involution $\iota_1$ given by swapping the two copies of $\widetilde E$. The quotient $(\widetilde E \cup_{\{x_i^\pm\}}\widetilde E) / \iota_1= \widetilde E$ has genus $1$. It follows that $\theta$ is well--defined and $\phi\circ \theta$ lands in $\overline{\Rcal\Bcal}_{g,0, 2m}^0$. These remarks also make it clear that diagram (\ref{eq: key commutative diagram with admissible covers}) is commutative.

\vskip 0.5em
As in \cite[Proof of Proposition 5]{zelm2018nontautBiell}, we are left to show that $\zeta : \Rcal_{1;k}' \rarr F$ is surjective on geometric points. 

\vskip 0.5 em 
By definition, a point in $F(\CC)$ is given by a triple $\bigl((\widetilde E_1/E_1, \widetilde E_2/ E_2), \widetilde S \rarr S \rarr T, \gamma\bigr)$ where $(\widetilde E_1/E_1, \widetilde E_2/E_2) \in (\Rcal_{1;k}'\times \Rcal_{1;k}')(\CC)$, $(\widetilde S \rarr S \rarr T) \in \overline{\Rcal\Adm}(g, 1)_{2m}^0$, and $\gamma$ is an isomorphism between $\chi(\widetilde E_1/E_1, \widetilde E_2/E_2)$ and $\phi(\widetilde S \rarr S \rarr T)$. Let $(\widetilde E, E) := \chi(\widetilde E_1/E_1, \widetilde E_2/E_2)$. Exactly as in \cite[Proof of Lemma 6]{zelm2018nontautBiell}, we see that the involution of $E \rarr \st(T)$ exchanges the two components $E_1$ and $E_2$ of $E$, and fixes all the nodes. Now, since $\Im(\phi) \subset\overline{\Rcal\Bcal}_{g, 0, 2m}^0$, there exists $\iota_1 : \widetilde E \rarr \widetilde E$ an involution of $\widetilde E / \st(T)$, different from the one over $E$, such that $g(\widetilde E/ \iota_1) = 1$. There are two possibilities for $\iota_1$. Assume by contradiction that $\iota_1$ fixes $\widetilde E_1$ and $\widetilde E_2$. In this case, since $g(\widetilde E_i/ \iota_1) \ge 1$ for $i = 1, 2$, it follows that $g(\widetilde E / \iota_1) \ge 2$, a contradiction.

\vskip 0.5em
Therefore, $\iota_1$ swaps $\widetilde E_1$ and $\widetilde E_2$. We claim that $\iota_1$ fixes all the nodes of $\widetilde E$. If not, there exists $q \in \widetilde E_\sing$ such that $\iota_1(q) \ne q$. This shows that $\widetilde E/\iota_1$ has a node at $[q \sim \iota_1(q)]$ and, since the partial normalisation of $\widetilde E/\iota_1$ at $[q \sim \iota_1(q)]$ is connected of genus $\ge 1$, we see that $\widetilde E /\iota_1$ has genus at least 2, again a contradiction. Thus, the involution $\iota_1$ has the desired properties. Surjectivity of $\zeta$ now follows as $[\widetilde E_1/E_1, x_1, \dots, x_k]$ maps to $\bigl((\widetilde E_1/E_1, \widetilde E_2/ E_2), \widetilde S \rarr S \rarr T, \gamma\bigr)$, and this concludes the proof.
\end{proof}

\subsection{Some lifting lemmas}

The following three lemmas will be the building blocks on which we will base our inductive procedure for the proof of Theorem \ref{thm: classes RBbar are non-taut}. In particular, they allow us to bypass the conditions $g + m$ is even required in the proof of Theorem \ref{thm: main non--tautology theorem}.
\begin{lemma}\label{lem: non tautological if add ramification point}
    Assume that $\left[\overline{\Rcal\Bcal}_{g, n, 2m}^0\right] \notin\R^\bullet(\overline\Rcal_{g, n + 2m})$. Then $\overline{\Rcal\Bcal}_{g, n+1, 2m}^0$ has non--tautological fundamental class. 
\end{lemma}
\begin{proof}
    Let $\pi: \ol{ \mc R}_{g;1+n+2m}\ra \ol{\mc R}_{g;n+2m}$ be the morphism that forgets the first marked point and then stabilises. Notice now that, since $2:1$ covers of an elliptic curve from a fixed genus $g$ curve correspond to bi--elliptic involutions $\iota\in \Aut(C)$ and those have, with multiplicity, $2g-2$ fixed points, when restricted to $\overline{\Rcal\Bcal}_{g, n+1, 2m}^0$, $\pi$ is finite. In particular, since by definition ${\pi}_{*}$ sends tautological classes to tautological classes and $\pi(\overline{\Rcal\Bcal}_{g, n+1, 2m}^0)= \overline{\Rcal\Bcal}_{g, n, 2m}^0$, the result follows.
\end{proof}
\begin{lemma}\label{lem: non tautological if add 2 conjugate points}
    Assume that $\left[\overline{\Rcal\Bcal}_{g, n, 2m}^0\right]\notin \R^\bullet\left(\ol{\mc R}_{g;n + 2m}\right)$. Then $\overline{\Rcal\Bcal}_{g, n, 2m+2}^0$ has non-tautological fundamental class. 
\end{lemma}
\begin{proof}
    Assume first that $n\leq 2g-3$ so that, by Lemma \ref{lem: non tautological if add ramification point}, under our assumptions $\left[\overline{\mc R\Bcal}_{g, n + 1, 2m}^0\right]$ is non-tautological. As in \cite{zelm2018nontautBiell}, consider the gluing morphism:
    \[
    \sigma:\ol{\mc R}_{g;n+2m+1}\times \overline{\Mcal}_{0,3}\ra \ol{\mc R}_{g,n+ 2m+2}
    \]
    Notice now that $\sigma$ restricted to $\overline{\Rcal\Bcal}_{g, n+1, 2m}^0$ factors through $\overline{\Rcal\Bcal}_{g, n, 2m+2}^0$. Indeed, keeping the same notation as above, consider the following commutative diagram:
    \[
    \begin{tikzcd}
        & \widetilde C \arrow[d,"\pi_\iota"] \arrow[dr, "\pi_{1}"] \arrow[dd, bend right=60,"\pi"']& \\
        & C=\widetilde C/\iota \arrow[d,"p"] & \widetilde C/\iota_1 \arrow[dl, "p_1"]\\
        & E
    \end{tikzcd}
    \]
    where $\iota_1$ is the involution of $\widetilde C$ such that $g(\widetilde C/\iota_1)=1$. Since $p_1$ is a non--constant morphism between curves of genus $1$, it is \'etale. This together with the fact that, by construction, $x_1$ is a ramification point of $p$, shows that $x_1^+$ and $x_1^-$ must be ramification points of $\pi_{1}$. In particular, $\iota_1$ fixes $x_1^+$ and $x_1^-$ and thus, calling $\iota_0$ the involution of $\mathbb P^1$ fixing $0$ and $\infty$, we get a well defined involution $\iota'_1$ on $\widetilde X$ by having it act as $\iota_1$ on $\widetilde C$ and as $\iota_0$ on its rational components. Quotienting by $\iota'_1$ then gives a Cornalba--admissible $2:1$ cover of a stable curve of genus $1$.
    
    \vskip 0.5 em
    Now, since for each $g,n,m,t$,
    \[
    \mathrm{codim}\left(\overline{\Rcal\Bcal}_{g, n+1, 2m}^t,\ol{\mc R}_{g;n+2m+1}\right)=\mathrm{codim}\left(\overline{\Rcal\Bcal}_{g, n, 2m+2}^t,\ol{\mc R}_{g;n+2m+2}\right),
    \]
    there is some positive scalar $\alpha\in \mathbb Q$ such that
    \[
    \left[\overline{\Rcal\Bcal}_{g, n+1, 2m}^0\right]=\alpha\sigma^*\left[\overline{\Rcal\Bcal}_{g, n, 2m+2}^0\right].
    \]
    Since $\sigma$ is a gluing morphism, by Proposition \ref{prop: pullback by chi admit tautological kunneth decomp}, if $\left[\overline{\Rcal\Bcal}_{g, n, 2m+2}^0\right]$ were to be tautological, the same would be true for $\left[\overline{\Rcal\Bcal}_{g, n+1, 2m}^0\right]$.
    \vskip 0.5 em
    Finally, as in \cite{zelm2018nontautBiell}, if $n=2g-2$, we may just repeat the previous reasoning to show that if $\left[\overline{\Rcal\Bcal}_{g, 2g-2, 2m}^0\right]$ is non-tautological so is $\left[\overline{\Rcal\Bcal}_{g, 2g-3, 2m+2}^0\right]$. At this point we may conclude by Lemma \ref{lem: non tautological if add ramification point}.
\end{proof}
\begin{remark}
    Notice that, also when $t\neq 0$, when we restrict the gluing morphism $\sigma$ to $\overline{\Rcal\Bcal}_{g, n+1, 2m}^t$ we again land in $\overline{\Rcal\Bcal}_{g, n, 2m+2}^t$. The only difference being the involution we have to take on the two copies of $\mathbb P^1$. If the two preimages of $x_1$ are not ramification points of $\pi_1$, swap the two components, otherwise proceed as in the proof of Lemma \ref{lem: non tautological if add ramification point}.
\end{remark}
\begin{lemma}\label{lem: non tautological if add node}
    Assume that $\left[\overline{\Rcal\Bcal}_{g, 1, 0}^0\right]\notin \R^\bullet\left(\ol{\mc R}_{g;1}\right)$. Then $\overline{\Rcal\Bcal}_{g + 1}^0$ has non-tautological fundamental class. 
\end{lemma}
\begin{proof}
    Consider the gluing morphism:
    \[\chi_{\mc R}: \ol{\mc R}_{g;1}\times \ol{\mc M}_{1,1}\ra \ol{\mc R}_{g+1}\]
    On closed points, this maps a tuple $\left((\widetilde C,p^+,p^-)\ra (C,p), (E,q)\right)$ to the cover $\widetilde X\ra X$ where the curve $X:= C \smallcoprod E/p\sim q$ and $\widetilde X$ is the curve obtained by attaching two copies of $E$ to $\widetilde C$ and identifying $p^\pm$ with $q$.
    \vskip 0.5 em
    Consider now the following commutative diagram whose horisontal arrows are the gluing morphisms and whose vertical arrows are the forgetful ones:
    \[
    \begin{tikzcd}
        \ol{\mc R}_{g;1}\times \ol{\mc M}_{1,1} \arrow[r,"\chi_{\mc R}"] \arrow[d,"(u\text{,}id)"] & \ol{\mc R}_{g+1} \arrow[d,"u"]\\
        \ol{\mc M}_{g,1}\times \ol{\mc M}_{1,1} \arrow[r,"\chi"] & \ol{\mc M}_{g+1}.
    \end{tikzcd}
    \]
    As noted in \cite[Lemma 10]{zelm2018nontautBiell}, calling $\left[\ol{\mc H}_{g-1,0,2},\ \ol{\mc M}_{1,1}\right]$ the fundamental class of the subvariety of $\ol{\mc M}_{g,1}\times \ol{\mc M}_{1,1}$ parametrising tuples of curves $\left((C,p),\ (E,q)\right)$ where $(C,q)$ is a pointed curve of genus $g$ with an irreducible component of genus $1$, $E'$ isomorphic to $E$ attached to a hyperelliptic stable curve $H$ at the point of $H$ that is conjugate under the hyperelliptic involution to $q$, there exist two positive scalars $\alpha,\beta\in \mathbb Q$ such that
    \begin{align}\label{eq: pullbak RB g+1}
    \chi_{\mc R}^*\left[\overline{\Rcal\Bcal}_{g+1}\right] = \alpha\left[\overline{\Rcal\Bcal}_{g;1}\times \ol{\mc M}_{1,1}\right]+\beta(u,id)^*\left[\ol{\mc H}_{g-1,0,2},\ \ol{\mc M}_{1,1}\right].
    \end{align}
    \vskip 0.5 em
    Notice now that, since a generic point of $(u,id)^{-1}\left[\ol{\mc H}_{g-1,0,2},\ \ol{\mc M}_{1,1}\right]$ corresponds to a tuple $\left((\widetilde C,p^\pm)\ra (C,p),(E,q)\right)$, where $C=H\cup_{\ol q} E$ with $H$ irreducible hyperelliptic curve of genus $g-1$, we have that $\chi_\mc R(\widetilde C/C, E)\notin \overline{\Rcal\Bcal}_{g+1}^0$. Indeed, since $\widetilde C\ra C$ is a Cornalba--admissible \'etale cover of degree $2$, there are only three possible cases for what it could be.

    \vskip 0.5em
    \textit{Case i)} $\widetilde C$ is given by two copies of $H$ attached to a non--trivial double cover $\widetilde E\ra E$ by identifying the two points over $\ol q$ in $\widetilde E$ to the two copies of $\ol q\in H$. 

    \vskip 0.5em
    \textit{Case ii)} $\widetilde C$ is given by a non--trivial cover of $H$ attached to a trivial double cover of $E$.

    \vskip 0.5em
    \textit{Case iii)} $\widetilde C$ is given by a non--trivial cover of $H$ attached at the two points $\ol q^+, \ol q^-$ lying over $\ol q$ to a non--trivial double cover $\widetilde E\ra E$.
    
    \vskip 0.5 em
    Now, calling $\widetilde X\ra X$ the image under $\chi_{\mc R}$ of $\left((\widetilde C,p^\pm)\ra (C,p),(E,q)\right)$ and $\pi:\widetilde X\ra E\cup_p \mathbb P^1$ the composition $\widetilde X\ra X\ra E\cup_p\mathbb P^1$, $\pi$ is unramified outside of the cover of $H$. In particular, there is no involution $\iota_1\in \mathrm{Aut}(\widetilde X)$ making the genus of $\widetilde X/\iota_1$ smaller than $2$. Combining this with (\ref{eq: pullbak RB g+1}), we immediately get that $\chi_{\mc R}^*\left[ \overline{\Rcal\Bcal}_{g + 1}^0\right] = \alpha\left[ \overline{\Rcal\Bcal}_{g;1}^0\times \ol{\mc M}_{1,1}\right]$.
    \vskip 0.5 em
    Now, if $\left[\overline{\Rcal\Bcal}_{g;1}^0\right]$ is not tautological, $\left[ \overline{\Rcal\Bcal}_{g;1}^0\times \ol{\mc M}_{1,1}\right]$ does not admit a Chow--Künneth decomposition. Hence, by Proposition \ref{prop: pullback by chi admit tautological kunneth decomp}, $\left[\overline{\Rcal\Bcal}_{g + 1}^0\right]$ is not tautological. 
\end{proof}

\subsection{The class \texorpdfstring{$[\overline{\Rcal\Bcal}_{g, 0, 2m}^0] \notin \R\H^\bullet(\rgmbar)$ for $g + m \ge 8$}{}}

The following proposition is the key step leading to the proof of Theorem \ref{thm: classes RBbar are non-taut}.
\begin{proposition}\label{prop: nontautology of RB bar and g+m = 8}
    Assume $g \ge 2$ and $g + m = 8$ and let $\chi:\overline\Rcal_{1;7} \times \overline\Rcal_{1;7}\ra \rgmbar$ be the glueing morphism. Then  $\chi^*[\overline{\Rcal\Bcal}_{g, 0, 2m}^0]$ has a non-trivial summand in $\H^{7,0}(\ol{\mc R}_{1,7})\otimes \H^{0,7}(\ol{\mc R}_{1,7})$. In particular, $[\overline{\Rcal\Bcal}_{g, 0, 2m}^0] \in \CH^\bullet(\rgmbar)$ is non--tautological. 
\end{proposition}
\begin{proof}
    Consider the map $\overline\Rcal_{1;7}' \times \overline\Rcal_{1;7}' \xrightarrow[]{\chi} \rgmbar' \xrightarrow[]{\phi} \rgmbar$ and suppose by contradiction that $[\overline{\Rcal\Bcal}_{g, 0, 2m}^0] \in \R^\bullet(\rgmbar)$. Pulling back to $\overline{\Rcal}_{1;7}' \times \overline\Rcal_{1;7}'$ and using Excision, we get that
    \[
    (\phi\circ \chi)^*[\overline{\Rcal\Bcal}_{g, 0, 2m}^0] = \alpha\cdot [\Delta] + [Z]
    \]
    where $\Delta \subset \overline{\Rcal}_{1;7}' \times \overline\Rcal_{1;7}'$ is the diagonal and $Z \subset \overline{\Rcal}_{1;7}' \times \overline\Rcal_{1;7}'$ is a cycle supported on $\partial({\Rcal}_{1;7}' \times \Rcal_{1;7}')$. Using gluing maps and Theorem \ref{thm: key basic result on R1m}, we see that all classes supported on $\partial({\Rcal}_{1;7}' \times \Rcal_{1;7}')$ admit tautological Künneth decomposition. All these put together imply that $[\Delta] \in \H^\bullet(\overline{\Rcal}_{1;7}' \times \overline\Rcal_{1;7}')$ admits tautological Künneth decomposition.

    \vskip 0.5em
    On the other hand, let $(e_i)_{i \in I}$ be a basis for $\H^\bullet(\Rcal_{1;7}')$ with dual basis $(e_i^\vee)_{i \in I}$. The class of the diagonal is given by
    \[
    [\Delta] = \sum_{i \in I} e_i \otimes e_i^\vee \in \H^{14}(\Rcal_{1;7}' \times \Rcal_{1;7}').
    \]
    By Theorem \ref{thm: key basic result on R1m}, $\H^{7, 0}(\overline\Rcal_{1;7}) \ne 0$ so that odd cohomology appears in the decomposition of $[\Delta]$. But this is contradiction with Proposition \ref{prop: pullback by chi admit tautological kunneth decomp}, finishing the proof.
\end{proof}

Proposition \ref{prop: nontautology of RB bar and g+m = 8} coupled with Lemma \ref{lem: non tautological if add ramification point}, Lemma \ref{lem: non tautological if add 2 conjugate points}, Lemma \ref{lem: non tautological if add node}, and Lemma \ref{lemma: non-vanishing hol forms on interior} immediately give Theorem \ref{thm: main non--tautology theorem}. Proceeding identically to the proof of Propositions \ref{prop: nontautology of RB bar and g+m = 8}, we obtain the following result.
\begin{corollary}\label{cor: coro on bad summand}
     Assume $g \ge 2$, $k+1:=g + m \geq  8$ even. Let $\chi:\overline\Rcal'_{1;k} \times \overline\Rcal'_{1;k}\ra \rgmbar$ be the glueing morphism. The class $\chi^*[\overline{\Rcal\Bcal}_{g, 0, 2m}^0]$ has a non-trivial summand in $\H^{k,0}(\ol{\mc R}_{1,k})\otimes \H^{0,k}(\ol{\mc R}_{1,k})$. 
\end{corollary}

\begin{proof}[Proof of Theorem \ref{thm: main non--tautology theorem}]
Suppose that $\left[\Rcal\Bcal_{g, 0, 2m}^0\right]\in \R^\bullet(\mc R_{g;2m})$ with $k +1 := g+m \ge 8$ even. By Excision, we have a collection of prime cycles $B_i$ of codimension $k + 1$ and supported on $\partial\mc R_{g;2m}$ such that
\[
\gamma :=\left[\overline{\Rcal\Bcal}_{g, 0, 2m}^0\right]+\sum_i [B_i]
\]
is tautological. Now, by Proposition \ref{prop: pullback by chi admit tautological kunneth decomp}, calling $\chi$ the gluing morphism $\ol{\mc R}'_{1;k}\times \ol{\mc R}'_{1;k}\ra \ol{\mc R}'_{g;2m}$, we have that $\chi^*\gamma$ admits a tautological K\"unneth decomposition, in particular, it has no summand in $\H^{k,0}(\ol{\mc R}_{1,k}')\otimes \H^{0,k}(\ol{\mc R}_{1,k}')$.

\vskip 0.5 em
However, since $\chi^*\left[\overline{\Rcal\Bcal}_{g, 0, 2m}\right]$ restricts on $\Rcal'_{1;k}\times \Rcal'_{1;k}$ to a non--zero multiple of the diagonal by Lemma \ref{lemma: pullback of biell cmp in Rgm is a multiple of the diagonal}, the Hodge--K\"unneth decomposition of the diagonal, Theorem \ref{thm: key basic result on R1m}, and Lemma \ref{lemma: non-vanishing hol forms on interior} together imply that $\chi^*\left[\overline{\Rcal\Bcal}_{g, 0, 2m}\right]$ has non--zero part in $\H^{k, 0}(\overline{\Rcal}'_{1;k}) \otimes \H^{0, k}(\overline{\Rcal}'_{1;k})$. It follows that the pullback $\chi^*\sum_i [B_i]$ has non--trivial contribution from $\H^{k, 0}(\overline{\Rcal}_{1;k}')\otimes \H^{0, k}(\overline{\Rcal}_{1;k}')$.

\vskip 0.5 em
Let $\Delta_{h:g-h}$  be the boundary component of $\ol{\mc R}'_{g;2m}$ generically parametrising stable Cornalba--admissible \'etale covers $\pi:\widetilde C\ra C$ where $C$ is the union of two irreducible components of genus $h$ and $g-h$ meeting at $p$ while $\widetilde C$ consists of the union of two non-trivial covers of these two components meeting at the preimages of $p$. Assume $B_i$ is supported on $\Delta_{h:g-h}$. Since no element in the image of $\chi|_{\Rcal_{1;k}'\times\Rcal_{1;k}'}$ does not have a separating node, $\chi^*[B_i]$ is supported on the boundary of $\ol{\mc R}'_{1;k}\times \ol{\mc R}'_{1;k}$. In particular it restricts trivially to the interior. Lemma \ref{lemma: non-vanishing hol forms on interior} then implies that it does not have a non--trivial summand in $\H^{k,0}(\ol{\mc R}'_{1,k})\otimes \H^{0,k}(\ol{\mc R}_{1,k}')$.

\vskip 0.5 em
Similarly, calling $\Delta_0^{\mathrm{ram}}\subset \ol{\mc R}_{g;2m}'$ the boundary divisor generically parametrising $2:1$ covers of smooth irreducible genus $g-1$ curves ramified at $2$ points, if $B_i$ were supported on $\Delta_0^{\mathrm{ram}}$, since any cover in the image of $\chi_{|_{{\mc R}'_{1;k}\times {\mc R}'_{1;k}}}$ is unramified, the same argument as above shows that $\chi^*[B_i]$ does not have a non--trivial summand in $\H^{k,0}(\ol{\mc R}_{1,k}')\otimes \H^{0,k}(\ol{\mc R}_{1,k}')$.

\vskip 0.5 em
Let $\Delta_0''\subset \ol{\mc R}'_{g;2m}$ be the boundary divisor given by the clutching map $\ol{\mc M}_{g-1,2m+2}\ra \ol{\mc R}'_{g;2m}$ and assume that $B_i$ were supported on $\Delta_0''$. Notice now that $\chi_{|_{{\mc R}'_{1;k}\times {\mc R}'_{1;k}}}$ does not intersect $\Delta''_0$. Indeed, by our description of the boundary divisors, any point $[\widetilde C\ra C]\in\Delta''_0$ with reducible base would have a source with at least $4$ components. Since any point of $\chi(\mc R'_{1;k}\times {\mc R}'_{1;k})$ is of the form $\widetilde C_1\cup \widetilde C_2\ra C_1\cup C_2$ with $\widetilde C_i$ smooth irreducible, $\chi^*B_i$ must then be supported on the boundary. Again Lemma \ref{lemma: non-vanishing hol forms on interior} then implies that it does not have a non--trivial summand in $\H^{k,0}(\ol{\mc R}_{1,k}')\otimes \H^{0,k}(\ol{\mc R}_{1,k}')$.

\vskip 0.5 em 
Let now $\Delta_0'\subset \ol{\mc R}_{g;2m}$ be the boundary divisor generically parametrising $2:1$ \'etale covers of irreducible nodal curves of geometric genus $g-1$ and assume that $B_i$ were supported on $\Delta_0'$. Factoring $\chi$ as
\[
\begin{tikzcd}
    \ol{\mc R}'_{1;k}\times \ol{\mc R}'_{1;k} \arrow[r,"\chi_1"] & \ol{\mc R}'_{g;2m+2} \arrow[r,"\chi_2"] & \ol{\mc R}'_{g;2m},
\end{tikzcd}\]
there would then exist cycles $A_i$ such that $[B_i]=\chi_{2,*}[A_i]$, so that
\[\chi^*[B_i]=\chi_1^*\left(c_1\left(\mc N_{\chi_2}\right)\left[A_i\right]\right)=  \chi_1^*\left(c_1\left(\mc N_{\chi_2}\right)\right)\chi_1^*\left[A_i\right].\]
Since, by Lemma \ref{lem: normal bundle}, $c_1\left(\mc N_{\chi_2}\right)$ is tautological, $\chi_1^*\left(c_1\left(\mc N_{\chi_2}\right)\right)$ is algebraic and hence of type $(1,1)$. In particular, $\chi^*[B_i]$ cannot have a non--trivial summand in $\H^{k,0}(\ol{\mc R}'_{1,k})\otimes \H^{0,k}(\ol{\mc R}'_{1,k})$.
\end{proof}

\bibliography{bibliography}
\bibliographystyle{amsalpha}

\end{document}